\documentclass[bj,preprint]{imsart}
\setattribute{journal}{name}{}
\usepackage{lmodern}
\usepackage{helvet}
\usepackage[T1]{fontenc}
\usepackage[latin9]{inputenc}
\usepackage{color}
\usepackage{amsthm}
\usepackage{amsmath}
\usepackage{amssymb}
\usepackage{graphicx}
\usepackage{esint}
\usepackage[numbers]{natbib}
\usepackage{bbm}
\usepackage{a4wide}

\makeatletter

\theoremstyle{plain}
\newtheorem{thm}{Theorem}
  \theoremstyle{plain}
  \newtheorem{prop}[thm]{Proposition}
  \theoremstyle{definition}
  
  \theoremstyle{plain}
  
  \theoremstyle{plain}
  \newtheorem{cor}[thm]{Corollary}
  \theoremstyle{plain}
  
  \theoremstyle{plain}
  \newtheorem{lem}[thm]{Lemma}

\makeatother

\begin{document}
\global\long\def\phi{\varphi}
\global\long\def\epsilon{\varepsilon}
\global\long\def\theta{\vartheta}
\global\long\def\E{\mathbb{E}}
\global\long\def\Var{\operatorname{Var}}
\global\long\def\Cov{\operatorname{Cov}}
\global\long\def\N{\mathbb{N}}
\global\long\def\Z{\mathbb{Z}}
\global\long\def\R{\mathbb{R}}
\global\long\def\F{\mathcal{F}}
\global\long\def\le{\leqslant}
\global\long\def\ge{\geqslant}
\global\long\def\MT{\ensuremath{\clubsuit}}
\global\long\def\1{\mathbbm1}
\global\long\def\d{\mathrm{d}}
\global\long\def\tr{\operatorname{tr}}
\global\long\def\diag{\operatorname{diag}}
\global\long\def\rank{\operatorname{rank}}
\global\long\def\Re{\operatorname{Re}}
\global\long\def\sign{\operatorname{sign}}
\global\long\def\P{\mathbb{P}}
\global\long\def\subset{\subseteq}
\global\long\def\supset{\supseteq}
\global\long\def\argmin{\operatorname*{arg\, min}}
\global\long\def\bull{{\scriptstyle \bullet}}
\global\long\def\supp{\operatorname{supp}}
\global\long\def\sgn{\operatorname{sign}}

%
%
\begin{frontmatter}

\title{Sparse covariance matrix estimation in high-dimensional deconvolution}
\runtitle{Covariance estimation in high-dimensional deconvolution}

\begin{aug}
\author{\fnms{Denis} \snm{Belomestny}\ead[label=e1]{denis.belomestny@uni-due.de}\thanksref{t1,t11}}, 
\author{\fnms{Mathias} \snm{Trabs}\ead[label=e2]{mathias.trabs@uni-hamburg.de}\thanksref{t2}}
\and
\author{\fnms{Alexandre B.} \snm{Tsybakov}\ead[label=e3]{alexandre.tsybakov@ensae.fr}\thanksref{t3}}

\runauthor{D. Belomestny, M. Trabs and A.B. Tsybakov}


\address[t1]{
Duisburg-Essen University, 
Faculty of Mathematics\\
Thea-Leymann-Str. 9
D-45127 Essen,
Germany}
\address[t11]{
National Research University 
Higher School of Economics 
\\
Shabolovka, 26, 119049 Moscow, Russia,\\
\phantom{E-mail: \ }
\printead{e1}
}

\address[t2]{Universit\"at Hamburg,
Faculty of Mathematics\\
Bundesstra{\ss}e 55,
20146 Hamburg,
Germany \\
\phantom{E-mail: \ }
\printead{e2}}

\address[t3]{
CREST, ENSAE, Universit\'e Paris-Saclay\\
5, avenue Henry Le Chatelier,
91120 Palaiseau,
France\\
\phantom{E-mail: \ }
\printead{e3}}

\end{aug}

\begin{abstract}
We study the estimation of the covariance matrix $\Sigma$ of a $p$-dimensional normal random vector based on $n$ independent observations corrupted by additive noise. Only a general nonparametric assumption is imposed on the distribution of the noise without any sparsity constraint on its covariance matrix. In this high-dimensional semiparametric deconvolution problem, we propose spectral thresholding estimators that are adaptive to the sparsity of $\Sigma$. We establish an oracle inequality for these estimators under model miss-specification and derive non-asymptotic minimax convergence rates that are shown to be logarithmic in $n/\log p$. We also discuss the estimation of low-rank matrices based on indirect observations as well as the generalization to elliptical distributions. The finite sample performance of the threshold estimators is illustrated in a numerical example. 
\end{abstract}



\end{frontmatter}

\section{Introduction}

One of the fundamental problems of multivariate data analysis is to estimate the covariance matrix $\Sigma\in\R^{p\times p}$ of a random vector $X\in\R^p$ based on independent and identically distributed (i.i.d.) realizations $X_1,\dots,X_n$  of $X$. An important feature of data sets in modern applications is  high dimensionality. Since it is well known that classical procedures fail if the dimension $p$ is large, various novel methods of high-dimensional matrix estimation have been developed in the last decade. However, an important question has not yet been settled: How can $\Sigma$ be estimated in a high-dimensional regime if the observations are corrupted by noise?

Let $X_1,\dots,X_n$ be i.i.d. random variables with multivariate normal distribution $\mathcal N(0,\Sigma)$. The maximum likelihood estimator of $\Sigma$ is the sample covariance
estimator 
\[
\Sigma_{X}^{*}:=\frac{1}{n}\sum_{j=1}^{n}X_{j}X_{j}^{\top}.
\]
The estimation error of $\Sigma_{X}^{*}$ explodes for large $p$. To overcome this problem, sparsity assumptions
can be imposed on $\Sigma$, reducing the effective number of parameters. The first rigorous studies of this idea go back to \citet{bickelLevina2008threshold, bickelLevina2008} and \citet{ElKaroui2008} who have assumed that most entries of $\Sigma$
are zero or very small. This allows for the construction of banding, tapering and thresholding estimators based on $\Sigma^{*}_X$, for which the dimension $p$ can grow exponentially in $n$. Subsequently, a rich theory has been developed in this direction including \citet{LamFan2009} who proposed a penalized pseudo-likelihood approach, \citet{CaiEtAl2010} who studied minimax optimal rates, \citet{caiZhou2012} studying the $\ell_1$ loss as well as \citet{rothmanEtAl2009} and \citet{CaiLiu2011} for more general threshold procedures and adaptation, to mention only the papers most related to the present contribution. For current reviews on the theory of large covariance estimation, we refer to \cite{CaiEtAl2016,FanEtAl2016}. Heading in a similar direction as noisy observations, covariance estimation in the presence of missing data has been recently investigated by \citet{lounici2014} as well as \citet{caiZang2016}.

Almost all estimators in the afore mentioned results build on the sample covariance estimator $\Sigma_{X}^{*}$. In this paper, we assume that only the noisy observations
\[
Y_{j}=X_{j}+\epsilon_{j},\quad j=1,\dots,n,
\]
are available, where the errors $\epsilon_{1},\dots,\epsilon_{n}$ are i.i.d. random vectors in $\R^p$ independent of $X_1,\dots,X_n$. Then the sample covariance estimator $\Sigma_{Y}^{*}$ 
is biased: 
\[
\E[\Sigma_{Y}^{*}]=\E\Big[\frac{1}{n}\sum_{i=1}^{n}Y_{i}Y_{i}^{\top}\Big]=\Sigma+\Gamma
\]
where $\Gamma=\E[\epsilon_{1}\epsilon_{1}^{\top}]$ is the covariance matrix of the errors. Assuming $\Gamma$ known to correct the bias is not very realistic. Moreover, for heavy tailed the errors $\epsilon_{j}$ that do not have finite second moments, $\Gamma$ is not defined and the argument based on $\Sigma_{Y}^{*}$ makes no sense.  
Several questions arising in this context will be addressed below:
\begin{enumerate}
\item How much information on the distribution of $\epsilon_{j}$ do we need to consistently estimate $\Sigma$?
\item Do we need finite second moments of $\epsilon_j$ and/or sparsity restrictions on $\Gamma$ to estimate $\Sigma$?
\item What is the minimax optimal rate of estimating $\Sigma$ based on noisy observations?
\end{enumerate}

If the covariance matrix $\Gamma$ of the errors exists and is known, the 
problem does not differ from the direct observation case, since $\Gamma$ can be simply subtracted from \(\Sigma_{Y}^{*}\). If $\Gamma$ can be estimated, for instance from a separate sample of the error distribution or from repeated measurements, we can proceed similarly. However, in the latter case, we need to assume that $\Gamma$ is sparse, since otherwise we cannot find a good estimator for large dimensions. 
Reducing our knowledge about $\epsilon_j$ further, we may only assume that the distribution of $\epsilon_j$ belongs to a given nonparametric class. This leads to a high-dimensional deconvolution model. The difference from standard deconvolution problems is that the density of $X_j$'s is a parametric object known up to a high-dimensional matrix parameter $\Sigma$. A related model in the context of stochastic processes has been recently studied by \citet{belomestnyTrabs2015}. Obviously, we need some assumption on the distribution of errors since otherwise $\Sigma$ is not identifiable as, for example, in the case of normally distributed $\epsilon_j$. It turns out that we do not need a sparse covariance structure for the error distribution and we can allow for heavy tailed errors without any moments. 

From the deconvolution point of view, it might seem surprising that $\Sigma$ and thus the distribution of $X_j$ can be estimated consistently without knowing or estimating the distribution of errors $\epsilon_j$, but as we will show it is possible. The price to pay for this  lack of information is in the convergence rates that turn out to be very slow - logarithmic in the sample size. In the pioneering works in one-dimensional case, \citet{matias2002}, \citet{ButuceaMatias2005} have constructed a variance estimator in deconvolution model with logarithmic convergence rate and a corresponding lower bound. In this paper, we provide a general multidimensional analysis of the minimax rates on the class of sparse covariance matrices. 

To replace the sample covariance matrix $\Sigma^*_Y$ by a deconvolution counterpart, we use some ideas from the literature on density deconvolution. Starting with \citet{carrollHall1988} and \citet{fan1991}, the deconvolution problem have been extensively studied. In particular, unknown (but inferable) error distributions have been analysed by \citet{neumann1997, delaigleEtAl2008, johannes2009} and \citet{delaigleHall2016} among others. For adaptive estimation with unknown error distribution we refer to \citet{comteLacour2011, kappusMabon2014,dattnerEtAl2016} and references therein. Almost all contributions to the deconvolution literature are restricted to a univariate model. Hence, our study contributes to the deconvolution theory by treating the multivariate case; in particular, our techniques for the lower bounds might be of interest. To our knowledge, only \citet{masry1993}, 
\citet{eckleBissantzDette2016},  
and \citet{LepskiWiller2017_1,LepskiWiller2017_2} have studied the setting of multivariate deconvolution. They deal with a different  problem, namely that of nonparametric estimation of the density of $X_j$ or its geometric features when the distribution of $\epsilon_j$ is known.

Applying a spectral approach, we construct an estimator for the covariance matrix assuming that $X_j$ are normally distributed and that the characteristic function $\psi$ of the distribution of $\epsilon_j$ decays slower than the Gaussian characteristic function. A similar idea in a one-dimensional deconvolution problem has been developed by \citet{butuceaEtAl2008}. The assumption $\big|\log |\psi(u)|\big|=o(|u|^2)$ as $|u|\to\infty$ implies identifiability of $\Sigma$ and allows us to construct an estimator $\hat\Sigma$, which is consistent in the maximal entry norm. Based on $\hat\Sigma$, we then construct hard and soft thresholding estimators $\hat\Sigma^H_\tau$ and $\hat\Sigma^S_\tau$, respectively, for sparse matrices. 
The sparsity is described by an upper bound $S$ on the $\ell_q$-norm, $q\in[0,2)$, of entries of $\Sigma$. 
We establish sparsity oracle inequalities for $\hat\Sigma^H_\tau$ and $\hat\Sigma^S_\tau$ when the estimation error is measured in the Frobenius norm. This choice of the norm is naturally related to the distance between two multivariate normal distributions. The oracle bounds reveal that the thresholding estimators adapt to the unknown sparsity $S$.
For the soft thresholding estimator we present an oracle inequality, which shows that the estimator adapts also to  approximate sparsity. 

Assuming that the characteristic function $\psi$ of $\epsilon_j$ satisfies  $\big|\log |\psi(u)|\big|=\mathcal O(|u|^\beta)$ for large $u\in\R^p$ and some $\beta\in[0,2)$, we prove the following upper bound on the estimation error in the Frobenius norm:
\begin{equation}\label{ora}
  \|\hat\Sigma_\tau^H-\Sigma\|\le CS^{1/2}\Big(\log\frac{n}{\log p}\Big)^{-(1-\beta/2)(1-q/2)}
\end{equation}
for some constant $C>0$ and with high probability. 
The dependence of this bound on the sparsity $S$ is the same as found by \citet{bickelLevina2008threshold} for the case direct observations; furthermore the well-known quotient $n/\log p$ drives the rate. However, the severely ill-posed nature of the inverse problem causes the logarithmic dependence of the rate on $n/\log p$. We also see that the estimation problem is getting harder if $\beta$ gets closer to 2 where it is more difficult to distinguish the signal from the noise. Furthermore, we establish a lower bound showing that the rate in \eqref{ora} cannot be improved in a minimax sense for $q=0$. Let us emphasise that our observations $Y_j$ are by definition not normally distributed. Therefore, the proof of the lower bound differs considerably from the usual lower bounds in high-dimensional statistics, which rely on Gaussian models. 

Covariance estimation is crucial in many applications where also observation errors appear. For instance, many portfolio optimization approaches rely on the covariance matrix of a possibly high number of assets where the financial data are typically perturbed due to bid-ask spreads, micro-structure noise etc. \cite{fanEtAl2012,TaoEtAl2013}. While in a high-frequency regime the observation noise can be handled by local averages, in a low-frequency situation, as daily closing prices, the denoising is more difficult and our deconvolution approach can be applied, cf. \cite{belomestnyTrabs2015}. Note that the dimension dependence in \cite{belomestnyTrabs2015} can be improved with our analysis for low-rank matrices. As another application the spatial empirical covariance matrices of climate data and their eigenvectors, called empirical orthogonal functions, are important spatio-temporal statistics. Naturally recordings of climate data, e.g. sea surface temperatures, may suffer from measurement errors \cite{cressieWikle2011} and should be taken into account.  Especially, sparse covariance structures appear in the problem of  spatio-temporal wind speed forecasting  taking into account the time series data of a target station and data of surrounding stations, see \cite{sanandaji2015low}. 

\medskip{}

This paper is organized as follows. In Section~\ref{sec:spectralEst} we construct and analyze the spectral covariance matrix estimator. In Section~\ref{sec:thres} the resulting thresholding procedures are defined and analyzed. In Section~\ref{sec:minimax} we investigate upper and lower bounds on the estimation error. In Section~\ref{sec:ext} some extensions of our approach are discussed including the estimation of low-rank matrices based on indirect observations as well as the generalization to elliptical distributions. The numerical performance of the procedure is illustrated in Section~\ref{sec:sim}. Longer and more technical proofs are postponed to Section~\ref{sec:proofs} and to the appendix.

\medskip{}

\emph{Notation:} For any $x\in\R^p$ and $q\in (0,\infty]$, the $\ell_q$-norm of $x$ is denoted by $|x|_q$ and we write for brevity $|x|:=|x|_2$. For $x,y\in\R^{p}$ the Euclidean scalar product
is written as $\langle x,y\rangle$. We denote by $I_p$ the $p\times p$ identity matrix, and by $\1_{\{\cdot\}}$ the indicator function.
For two matrices $A,B\in\R^{p\times p}$ the Frobenius scalar product
is given by $\langle A,B\rangle:=\tr(A^{\top}B)$ inducing the Frobeninus
norm $\|A\|:=\sqrt{\langle A,A\rangle}$. The nuclear norm is denoted by $\|A\|_1:=\tr(\sqrt{ A^\top A})$ and the spectral norm by $\|A\|_\infty:=\sqrt{\lambda_{max}(A^\top A)}$, where $\lambda_{max}(\cdot)$ stands for the maximal eigenvalue. For $A\in\R^{p\times p}$ and $q\in[0,\infty]$
we denote by $|A|_{q}$ the $\ell_{q}$-norm of the entries of the matrix if $q>0$ and the number of non-zero entries for $q=0$. We write $A>0$ or $A\ge 0$ if the matrix $A\in\R^{p\times p}$ is positive definite or semi-definite.
We denote by $\P_{\Sigma,\psi}$ the joint distribution of $Y_1,\dots,Y_n$ when the covariance matrix of $X_j$ is $\Sigma$ and the characteristic function of the noise $\epsilon_j$ is $\psi$. We will write for brevity $\P_{\Sigma,\psi}=\P$ if there is no ambiguity.

\section{Spectral covariance estimators}\label{sec:spectralEst}


Let $\psi$ denote the characteristic function of error distribution:
\[
\psi(u)=\E\big[e^{i\langle u,\epsilon_{1}\rangle}\big],\quad u\in\R^{p}.
\]
Then the characteristic function of $Y_{j}$ is given by
\[
\phi(u):=\E\big[e^{i\langle u,Y_{j}\rangle}\big]=\exp\big(-\frac{1}{2}\langle u,\Sigma u\rangle+\log\psi(u)\big),\quad u\in\R^{p}.
\]
Here and throughout we assume that $\psi(u)\neq0$ and we use the distinguished logarithm, cf. \cite[Lemma 7.6]{sato1999levy}. This assumption is standard in the literature on deconvolution. Allowing for some zeros of $\psi$ has been studied in \cite{Meister2008, delaigleMeister2011}. Note that our estimation procedure defined below does not rely on all $u$ in $\R^d$, but uses only $u$ with a certain radius $|u|$. 

The canonical estimator for the characteristic function $\phi$ is the empirical characteristic function
\begin{equation*}
\phi_{n}(u):=\frac{1}{n}\sum_{j=1}^{n}e^{i\langle u,Y_{j}\rangle},\quad u\in\R^{p}.
\end{equation*}
Since $\phi_n(u)$ concentrates around $\phi(u)$ with rate $\sqrt n$, we have $\phi_n(u)\neq0$ with overwhelming probability for sufficiently large frequencies $u$ ensuring $|\phi(u)|\ge C\sqrt{(\log(ep))/n}$ for some constant $C>1$ (see Lemma~\ref{lem:concentrationPhi} and Corollary~\ref{cor:concStochErr}). In this case $\log \phi_n(u)$ is well defined. On the unlikely event $\{\phi_n(u)=0\}$, we may set $\log \phi_n(u):=0$.

Arguing similarly to \citet{belomestnyTrabs2015}, we consider the identity
\begin{equation}
\frac{\log\phi_{n}(u)}{|u|^{2}}=-\frac{\langle u,\Sigma u\rangle}{|u|^{2}}+\frac{\log\psi(u)}{|u|^{2}}+\frac{\log\phi_{n}(u)-\log\phi(u)}{|u|^{2}},\quad u\in\R^{p}\setminus\{0\}.
\label{eq:regression}
\end{equation}
Both sides are normalized by $|u|^{2}$ being the order of the leading term $\langle u,\Sigma u\rangle$.
While the left-hand side of \eqref{eq:regression} is a statistic based on the observations $Y_1,\dots,Y_n$, the first term on the right-hand side encodes the parameter of interest,
namely the covariance matrix $\Sigma$. The second term is a deterministic error due to the unknown distribution of $\epsilon_j$. If $|\log \psi(u)|=o(|u|^2)$, i.e., the error distribution is less smooth than the normal distribution, the deterministic error vanishes as $|u|\to\infty$. The third term in \eqref{eq:regression} is a stochastic error term. Using the first order approximation we get
\begin{equation}\label{eq:linearisation}
  \log\phi_{n}(u)-\log\phi(u)=\log\Big(\frac{\phi_n(u)-\phi(u)}{\phi(u)}+1\Big)\approx\frac{\phi_n(u)-\phi(u)}{\phi(u)}.
\end{equation}
The latter expression resembles the estimation error in classical deconvolution problems. However, there is a difference since here in the denominator we have $\phi(u)$ rather than the characteristic function of the distribution of errors.  A similar structure was detected in the statistical analysis of low-frequently observed L\'evy processes by \citet{belomestnyReiss2006}. Following \cite{belomestnyReiss2006}, one can call this type of problems \emph{auto-deconvolution} problems. Since $|\phi(u)|=e^{-\langle u,\Sigma u\rangle/2}|\psi(u)|$, and we assume that $|\log \psi(u)|=o(|u|^2)$, the stochastic error grows exponentially in $|u|$. Thus, the estimation problem is severely ill-posed even in one-dimensional case.

These remarks lead us to the conclusion that $\Sigma$ can be estimated consistently without any particular knowledge of the error distribution as soon as $|\log \psi(u)|=o(|u|^2)$, and the spectral radius $|u|$ in \eqref{eq:regression} is chosen to achieve a trade-off between the stochastic and deterministic errors. To specify more precisely the  condition $|\log \psi(u)|=o(|u|^2)$, it is convenient to consider, for any $\beta\in(0,2)$ and $T>0$, the following nonparametric class of functions $\psi$:
\begin{align*}
 & \mathcal{H}_{\beta}(T) :=\big\{\psi\mbox{ characteristic function on }\R^{p}:\big|\log |\psi(u)| \big|\le {T}\big(1+|u|_\beta^{\beta}\big),\, u\in\R^{p}\big\}.
\end{align*}
Note that  $\big|\log |\psi(u)| \big|= \log (1/|\psi(u)|)$ since $|\psi(u)|\le 1$. Therefore, the condition that determines the class $\mathcal{H}_{\beta}(T)$ can be written as the lower bound $|\psi(u)|  \ge \exp\big(-{T}\big(1+|u|_\beta^{\beta}\big)\big)$.
If the characteristic function of $\epsilon_j$ belongs to $\mathcal H_\beta(T)$, the decay $|u|_\beta^\beta$ for some $\beta<2$ of the characteristic exponent allows for separating the normal distribution of $X_j$ from error distribution for large $|u|$. The decay rate $\beta$ determines the ill-posedness of the estimation problem. Noteworthy, we require neither sparsity restrictions on the joint distribution of $(\epsilon_1,\dots,\epsilon_n)$ nor moment conditions of these random variables.

A typical representative in the class $\mathcal H_\beta$ is a characteristic function of a vector of independent $\beta$-stable random variables. In the case of identically distributed marginals, it has the form $\psi(u)=\exp(-\sigma|u|_\beta^\beta),u\in\R^p,$ for some parameter $\sigma>0$. A related example with correlated coefficients is a $p$-dimensional stable distribution with characteristic function $\psi(u)=\exp(-\sigma|u|_2^\beta)$ (note that $|u|_2^\beta\le|u|_\beta^\beta$). Recalling that stable distributions can be characterized as limit distributions of normalized sums of independent random variables and interpreting $\epsilon_j$ as accumulation of many small measurement errors, suggests that these examples are indeed quite natural.

If $\psi\in\mathcal H_\beta(T)$, the deterministic error term in \eqref{eq:regression} is small for large values of $|u|$. We will choose $u$ in \eqref{eq:regression} in the form $Uu^{(i,j)}$ where $U>0$ is large, and $u^{(i,j)}$ are $p$-dimensional unit vectors defined by 
\begin{equation}
u^{(i,i)} :=u^{(i)}:=(\1_{\{i=k\}})_{k=1,\dots,p}\quad\mbox{ and }\quad u^{(i,j)} :=\frac{1}{\sqrt{2}}\big(u^{(i)}+u^{(j)}\big)\mbox{ for }i\neq j.\label{eq:uij}
\end{equation}
Using the symmetry of $\Sigma=(\sigma_{i,j})_{i,j=1,\dots,p}$, we obtain
\begin{align*}
\langle u^{(i)},\Sigma u^{(i)}\rangle & =\sigma_{i,i}\quad\mbox{and}\quad\langle u^{(i,j)},\Sigma u^{(i,j)}\rangle=\sigma_{i,j}+\frac{\sigma_{i,i}+\sigma_{j,j}}{2}
\end{align*}
for any $i,j\in\{1,\dots,p\}$ with $i\neq j$. Motivated by \eqref{eq:regression} applied to $Uu^{(i,j)}$ for some spectral radius $U>0$, we introduce the \emph{spectral covariance estimator}:
\begin{align}
\hat{\Sigma}=(\hat{\sigma}_{i,j})_{i,j=1,\dots p}\quad\mbox{with}\quad\hat{\sigma}_{i,j} & :=\begin{cases}
-\frac{1}{U^{2}}\Re\big(\log\phi_{n}(Uu^{(i)})\big), & \mbox{if }i=j,\\
-\frac{1}{U^{2}}\Re\big(\log\phi_{n}(Uu^{(i,j)})\big)-\frac{1}{2}(\hat{\sigma}_{i,i}+\hat{\sigma}_{j,j}),\quad & \mbox{if }i\neq j.
\end{cases}\label{eq:SigmaHat}
\end{align}
Equivalently, we can write $\Re(\log\phi_n(u))=\log|\phi_n(u)|$ for any $u\in\R^p$ with $|\phi_n(u)|\neq 0$. Since $\phi_n(u)$ concentrates around $\phi(u)$, cf. Lemma~\ref{lem:concentrationPhi}, we have $\phi_n(u)\neq0$ with high probability if $\phi(u)\neq0$. 

The spectral covariance estimator $\hat\Sigma$ can be viewed as a counterpart of the classical sample covariance matrix for the case of indirect observations. 
The entries $\hat\sigma_{i,j}$ of $\hat\Sigma$ enjoy the following concentration property. 
\begin{thm}\label{thm:ConcentrationSigma*} 
  Assume that $|\Sigma|_\infty\le R$, and $\psi\in\mathcal H_\beta(T)$ for some $\beta,R,T>0$. Let $\gamma>\sqrt{2}$  and \(U\ge1\) satisfy $8\gamma\sqrt{(\log (ep))/n}<e^{-RU^{2}-3TU^\beta}$. Set 
  \begin{equation}\label{eq:tau}
    \tau(U)=6\gamma\frac{e^{RU^{2}+3TU^\beta}}{U^2}\Big(\frac{\log (ep)}{n}\Big)^{1/2}+3TU^{-2+\beta}.
  \end{equation}
  Then, for any $\tau\ge\tau(U)$,
  \[
  \P_{\Sigma,\psi}\big(|\hat{\sigma}_{i,j}-\sigma_{i,j}|<\tau\big)\ge1-12(ep)^{-\gamma^{2}}\quad\mbox{and}\quad\P_{\Sigma,\psi}\Big(\max_{i,j=1,\dots,p}|\hat{\sigma}_{i,j}-\sigma_{i,j}|<{\tau}\Big)\ge1- c_*p^{2-\gamma^{2}}
  \]
  where $c_*= 12e^{-\gamma^{2}}$.
\end{thm}

\begin{proof}
  Set $S(u)=\Re(\log\phi_{n}(u)-\log\phi(u))$. Using \eqref{eq:regression} we obtain, for all  $i,j=1,\dots,p$,
  \begin{align*}
  |\hat{\sigma}_{i,i}-\sigma_{i,j}| & \le U^{-2}|S(Uu^{(i,j)})|+U^{-2}\big|\log|\psi(Uu^{(i,j)})|\big|\\
  &\le U^{-2}|S(Uu^{(i,j)})|+U^{-2}\max_{i\in\{1,\dots,p\}}\big|\log|\psi(Uu^{(i,j)})|\big|.
  \end{align*}
  For $U\ge 1$ the last summand in this display is bounded uniformly by $3TU^{-2+\beta}$ on the class $\mathcal H_\beta(T)$.  This remark and Corollary~\ref{cor:concStochErr} in Section~\ref{sec:proofSigma}
  imply that
  \[
     \P\Big(|\hat{\sigma}_{i,j}-\sigma_{i,j}|\ge\frac{6\gamma\sqrt{\log (ep)}}{\sqrt{n}U^{2}\min_{i,j\in\{1,\dots,p\}}|\phi(Uu^{(i,j)})|}+3TU^{-2+\beta}\Big)\le 12(ep)^{-\gamma^{2}}
  \]
  if the condition $\gamma\sqrt{(\log (ep))/n}<|\phi(Uu^{(i,j)})|/8$ is satisfied for all $i,j$.
  Note that for any $i,j=1,\dots,p$, and any $\psi\in\mathcal H_\beta(T)$,
  \begin{align*}
    |\phi(Uu^{(i,j)})| & =\exp\Big(-\frac{U^{2}\langle u^{(i,j)},\Sigma u^{(i,j)}\rangle}{2}+\Re\log\psi(Uu^{(i,j)})\Big)\\
    & \ge\exp\big(-U^{2}\big(|\Sigma|_\infty+3TU^{\beta-2}\big)\big).
  \end{align*}
  Therefore, for $\gamma$ and $U$ satisfying the conditions of the theorem,
  \[
  \P\Big(|\hat{\sigma}_{i,j}-\sigma_{i,j}|\ge 6\gamma\frac{e^{RU^{2}+3TU^\beta}}{U^2}\Big(\frac{\log (ep)}{n}\Big)^{1/2}+3TU^{-2+\beta}\Big)\le 12(ep)^{-\gamma^{2}}.
  \]
  A union bound concludes the proof.
\end{proof}
The first term in $\tau(U)$ is an upper bound for the stochastic error. We recover the familiar factor $\sqrt{(\log p)/n}$ which is due to a sub-Gaussian bound on the maximum of the $p^2$ entries $(\hat\sigma_{i,j})$. The term $\exp(RU^2+3TU^\beta)$ is an upper bound for $\phi(u)^{-1}$ appearing in the linearization \eqref{eq:linearisation}. Note that for $\beta< 2$ this bound can be written as $\exp(RU^2(1+o(1)))$ for $U\to\infty$. This suggests the choice of spectral radius in the form $U_*=c\sqrt{\log(n/\log(ep))}$ for some sufficiently small constant $c>0$. The second term in \eqref{eq:tau} bounds the deterministic error and determines the resulting rate $U_*^{-2+\beta}=\mathcal O((\log(n/\log(ep)))^{-1+\beta/2})$, cf. Theorem~\ref{cor:no}.

\section{Thresholding}\label{sec:thres}

Based on the spectral covariance estimator, we can now propose estimators of high-dimensional sparse covariance matrices. We consider the following {sparsity classes} of matrices:
\begin{align}
\mathcal{G}_{0}(S, R) & :=\Big\{\Sigma>0:\Sigma=\Sigma^{\top},|\Sigma|_{0}\le S,\,|\Sigma|_\infty\le R\Big\}\quad\mbox{and}\quad\label{eq:class}\\
\mathcal{G}_{q}(S, R) & :=\Big\{\Sigma>0:\Sigma=\Sigma^{\top},|\Sigma|_{q}^{q}\le S,\,|\Sigma|_\infty\le R\Big\}\quad\mbox{for }q\in(0,2),\nonumber 
\end{align}
where $S>0$ denotes the sparsity parameter and $R>0$ bounds the largest entry of $\Sigma$. We also consider larger classes $\mathcal{G}_{q}^*(S, R)$ that differ from $\mathcal{G}_{q}(S, R)$ only in that the condition $\Sigma>0$ is dropped. Note that $S\ge p$ for the classes $\mathcal{G}_{q}(S, R)$, since otherwise the condition $\Sigma>0$ does not hold. This restriction on $S$ does not apply to the classes $\mathcal{G}_{q}^*(S, R)$, for which the unknown effective dimension of $\Sigma$ can be smaller than $p$. However, for the classes $\mathcal{G}_{q}^*(S, R)$, the overall model remains, in general,  $p$-dimensional since the distribution of the noise can be supported on the whole space~$\R^p$.

The sparsity classes considered by \citet{bickelLevina2008threshold} and in many subsequent papers are given by 
\[
  \mathcal U_q(s,R):=\Big\{\Sigma>0:\Sigma=\Sigma^{\top},\max_i\sum_{j=1}^p|\sigma_{i,j}|^q\le s,\max_i\sigma_{i,i}\le R\Big\}
\]
for $s,R>0$, $q\in(0,1)$ and with the usual modification for $q=0$. We have $\mathcal U_q(s,R)\subset \mathcal{G}_{q}(sp, R)$, so that our results can be used to obtain upper bounds on the risk for the classes $\mathcal U_q(s,R)$. 

Based on the spectral covariance estimator, we define the \emph{spectral hard thresholding estimator} for \(\Sigma\) as 
\begin{align}\label{eq:hardtheshold}
\hat{\Sigma}_{\tau}^{H}: & =(\hat{\sigma}_{i,j}^{H})_{i,j=1,\dots,p}\quad\mbox{with}\quad\hat{\sigma}_{i,j}^{H}:=\hat{\sigma}_{i,j}\1_{\{|\hat\sigma_{i,j}|>\tau\}},
\end{align}
for some threshold value $\tau>0$. 
The following theorem gives an upper bound on the risk of this estimator in the Frobenius norm. 
\begin{thm}\label{thm:hardThreshold} 
  Let $R,T, S>0$, $\beta\in [0,2)$, and $q\in[0,2)$. Let $\tau(U)$ be defined in \eqref{eq:tau} with parameters $\gamma>\sqrt{2}$  and \(U\ge1\) satisfying  $8\gamma\sqrt{(\log (ep))/n}\le e^{-RU^{2}-3TU^\beta}$. Then   
  \[
  \sup_{\Sigma\in\mathcal{G}_{q}^*(S,R),\psi\in \mathcal{H}_{\beta}(T)}\P_{\Sigma,\psi}(\|\hat{\Sigma}_{\tau}^{H}-\Sigma\|\ge 3S^{1/2}\tau^{1-q/2})\le c_*p^{2-\gamma^{2}}
  \]
  provided that $\tau\ge\tau(U)$ for $q=0$, and $\tau\ge2\tau(U)$ for $q\in (0,2)$. Here, $c_*= 12e^{-\gamma^{2}}$.
\end{thm}
\begin{proof} 
First, consider the case $q=0$ and $\tau\ge\tau(U)$. In view of Theorem~\ref{thm:ConcentrationSigma*}, the event $\mathcal A = 
\big\{ \max_{i,j=1,\dots,p}|\hat{\sigma}_{i,j}-\sigma_{i,j}|<{\tau}\big\}$ is of probability at least $1-c_*p^{2-\gamma^{2}}$ for all $\tau\ge\tau(U)$.  On $\mathcal A$ we have the inclusion $\big\{ j: |\hat{\sigma}_{i,j}|>{\tau}\big\}\subseteq \big\{ j: {\sigma}_{i,j}\ne 0\big\} $, so that $|\hat\Sigma^H|_0\le |\Sigma|_0$. Therefore, on the event~$\mathcal A$, 
$$
\|\hat{\Sigma}_{\tau}^{H}-\Sigma\|^2\le |\hat{\Sigma}_{\tau}^{H}-\Sigma|_0 \,|\hat{\Sigma}_{\tau}^{H}-\Sigma|_\infty^2\le 2|\Sigma|_0 \,|\hat{\Sigma}_{\tau}^{H}-\Sigma|_\infty^2\le 2S |\hat{\Sigma}_{\tau}^{H}-\Sigma|_\infty^2.
$$
Note that, again on $\mathcal A$, we have $|\hat{\Sigma}_{\tau}^{H}-\Sigma|_\infty\le |\hat{\Sigma}_{\tau}^{H}-\hat\Sigma|_\infty+ |\hat{\Sigma}-\Sigma|_\infty\le 2\tau$. Combining this with the last display implies the assertion of the theorem for $q=0$. 
  
  Consider now the case $q\in (0,2)$ and $\tau\ge2\tau(U)$. We use the following elementary fact: If $|y-\theta|\le r$ for some $y,\theta\in \R$ and $r>0$, then $|y\1_{\{|y|>2r\}}-\theta|\le 3\min\{|\theta|,r\}$ (cf.\cite{TsybakovStFlour2013}). Taking $y=\hat{\sigma}_{i,j}$, $\theta=\sigma_{i,j}$, and $r=\tau/2$, and using Theorem~\ref{thm:ConcentrationSigma*} we obtain that, 
  on the event of probability at least $1-c_*p^{2-\gamma^{2}}$, 
  $$
  |\hat{\sigma}_{i,j}^{H}-\sigma_{i,j}|\le 3\min\{|\sigma_{i,j}|,\tau/2\}, \quad i,j=1,\dots,p. 
  $$
 Thus, for any $q\in(0,2)$, with probability at least $1-c_*p^{2-\gamma^{2}}$,
  \begin{align*}
  \|\hat{\Sigma}_{\tau}^{H}-\Sigma\|^{2} & =\sum_{i,j}(\hat{\sigma}_{i,j}^{H}-\sigma_{i,j})^{2}\le 9\sum_{i,j}\min\{\sigma_{i,j}^{2},\tau^{2}/4\}\le 9(\tau/2)^{2-q}|\Sigma|_{q}^{q}\le 9\tau^{2-q} S.
  \end{align*}
 Since all bounds hold uniform in $\Sigma\in\mathcal{G}_{q}^*(S,R)$ and $\psi\in \mathcal{H}_{\beta}(T)$, the theorem is proven.
\end{proof}
In the direct observation case where $\epsilon_j=0$ we have $\psi(u)=1$ for all $u\in\R^p$, so that the deterministic error term in \eqref{eq:tau} disappears. In this case, $U$ can be fixed and the threshold can be chosen as a multiple of $\sqrt{(\log p)/n}$, analogously to \cite{bickelLevina2008threshold}. Together with the embedding $\mathcal U_q(s,R)\subset \mathcal{G}_{q}(sp, R)$, we recover Theorem~2 from \citet{bickelLevina2008threshold}. In Section~\ref{sec:minimax} we will discuss in detail the optimal choice of the spectral radius and the threshold in the presence of noise.



The \emph{spectral soft thresholding estimator} is defined as 
\[
\hat{\Sigma}_{\tau}^{S}:=(\hat{\sigma}_{i,j}^{S})_{i,j=1,\dots,p}\quad\mbox{with}\quad\hat{\sigma}_{i,j}^{S}:=\sign(\hat{\sigma}_{i,j})\big(|\hat{\sigma}_{i,j}|-\tau\big)_{+}
\]
with some threshold ${\tau}>0$. It is well known, cf., e.g. \cite{TsybakovStFlour2013}, that  
\begin{equation}
\hat{\Sigma}_{\tau}^{S}=\argmin_{A\in\R^{p\times p}}\big\{|A-\hat{\Sigma}|_{2}^{2}+2\tau|A|_{1}\big\}.\label{eq:soft}
\end{equation}
Adapting the proof of Theorem~2 in \citet{rigolletTsybakov2012},
we obtain the following oracle inequality, which is sharp for $q=0$ and looses a factor $2$ otherwise. 
\begin{thm}\label{thm:oracleTheshold}
Assume that $|\Sigma|_\infty\le R$, and $\psi\in\mathcal H_\beta(T)$ for some $\beta,R,T>0$. Let $\tau\ge \tau(U)$ where $\tau(U)$ is defined in \eqref{eq:tau} with parameters $\gamma>\sqrt{2}$  and \(U\ge1\) such that  $8\gamma\sqrt{(\log (ep))/n}\le e^{-RU^{2}-3TU^\beta}$.
   Then,   \begin{equation}
    \|\hat{\Sigma}_{\tau}^{S}-\Sigma\|^{2}\le\min_{A\in\R^{p\times p}}\Big\{\|A-\Sigma\|^{2}+(1+\sqrt{2})^{2}\tau^{2}|A|_{0}\Big\}
  \label{eq:oracle0}
  \end{equation}
  with probability at least $1-c_*p^{2-\gamma^{2}}$ where $c_*=12 e^{-\gamma^{2}}$. For any $q\in(0,2)$ we have, with probability at least $1-c_*p^{2-\gamma^{2}}$,
  \begin{equation}
  \|\hat{\Sigma}^S_{\tau}-\Sigma\|^{2}\le\min_{A\in\R^{p\times p}}\Big\{2\|A-\Sigma\|^{2}+c(q)\tau^{2-q}|A|_{q}^{q}\Big\}\label{eq:oracle}
  \end{equation}
  where $c(q)>0$ is a constant depending only on $q$. 
  \end{thm}
\begin{proof}
  Starting from the characterization (\ref{eq:soft}), we use Theorem~2   by \citet{koltchinskii2011nuclear}. To this end, we write 
  $
  \hat{\sigma}_{i,j}=\sigma_{i,j}+\xi_{i,j},\ i,j\in\{1,\dots,p\},
  $
  where $\xi_{i,j}$ are random variables with exponential concentration around zero due to Theorem~\ref{thm:ConcentrationSigma*}. Observing
  $\hat{\sigma}_{i,j}$ is thus a sequence space model in dimension $p^{2}$ and a special case of the trace regression model $Y_j=\tr(Z_{i,j}^\top A_0)+\xi_{i,j}$
  considered in \cite{koltchinskii2011nuclear}. Namely, $A_{0}$ is the diagonal matrix with diagonal entries $\sigma_{i,j}$
  and $Z_{i,j}$ are diagonalisations of the canonical basis in $\R^{p\times p}$. In particular, Assumption~1 in \citep{koltchinskii2011nuclear} is satisfied for $\mu=p$, i.e., $\|B\|_{L_{2}(\Pi)}^{2}=p^{-2}|B|_{2}^{2}$ where we use the notation of \citep{koltchinskii2011nuclear}. Note also that the rank of a diagonal matrix $B$ is equal to the number of its non-zero
  elements. Consequently, Theorem~2 in \citep{koltchinskii2011nuclear} yields with $\lambda=\frac{2\tau}{p^{2}}$ that 
  \[
  |\hat{\Sigma}_{\tau}^{S}-\Sigma|_{2}^{2}\le\min_{A\in\R^{p\times p}}\Big\{|A-\Sigma|_{2}^{2}+(1+\sqrt{2})^{2}{\tau}^{2}|A|_{0}\Big\}
  \]
  on the event that $\mathcal A=\{\max_{i,j}|\hat{\sigma}_{i,j}-\sigma_{i,j}|<\tau\}$. 
  To estimate the probability of $\mathcal A$,
  we apply Theorem~\ref{thm:ConcentrationSigma*}. Inequality \eqref{eq:oracle} follows from \eqref{eq:oracle0} using the same argument as in Corollary~2 of \citep{rigolletTsybakov2012}.
\end{proof}
This theorem shows that the soft thresholding estimator allows for estimating matrices that are a not exactly sparse but can be well approximated by a sparse matrix.
Choosing $A=\Sigma$ in the oracle inequalities  \eqref{eq:oracle0}  and (\ref{eq:oracle}) we obtain the following corollary analogous to Theorem~\ref{thm:hardThreshold}. 
\begin{cor}\label{cor:soft}
Let $R,T, S>0$, $\beta\in (0,2)$, and $q\in[0,2)$. Let $\tau\ge \tau(U)$ where $\tau(U)$ is defined in \eqref{eq:tau} with parameters $\gamma>\sqrt{2}$  and \(U\ge1\) such that  $8\gamma\sqrt{(\log (ep))/n}\le e^{-RU^{2}-3TU^\beta}$. Then   
   \[
  \sup_{\Sigma\in\mathcal{G}_{q}^*(S,R),\psi\in \mathcal{H}_{\beta}(T)}\P_{\Sigma,\psi}(\|\hat{\Sigma}_{\tau}^{S}-\Sigma\|\ge CS^{1/2}\tau^{1-q/2})\le c_*p^{2-\gamma^{2}}
  \]
  where $C=1+\sqrt{2}$ for $q=0$, and $C=\sqrt{c(q)}$ for $q\in(0,2)$.
\end{cor}

\section{Minimax optimality}\label{sec:minimax}


In this section, we study minimax optimal rates for the estimation of $\Sigma$ on the class $\mathcal{G}_{q}(S,R)\times\mathcal{H}_{\beta}(T)$. We first state an upper bound on the rate of convergence of the hard thresholding estimator in this high-dimensional semiparametric problem. It is an immediate consequence of Theorem~\ref{thm:hardThreshold}. 
Due to Corollary~\ref{cor:soft}, the result directly carries over to the soft thresholding estimator.

\begin{thm}\label{cor:no}
Let $R,T, S>0$, $\beta\in (0,2)$, and $q\in[0,2)$.
  For $\gamma>\sqrt{2}$, set
  \begin{align}\label{eq:U*Tau}
    U_{*}=\sqrt{\frac{1}{4R}\log\frac{n}{64 \gamma^2 \log (ep)}}.
  \end{align}
 Let $n$ be large enough such that $U_*\ge (\frac{3T}{R})^{1/(2-\beta)}\vee (\bar c/T)^{1/\beta} \vee 1$ for some numerical constant $\bar c>0$. Then for any 
 $\tau\ge \tau(U_*)$ where $\tau(\cdot)$ is defined in \eqref{eq:tau} we have
 \begin{align}
  \sup_{(\Sigma,\psi)\in\mathcal{G}_{q}(S,R)\times\mathcal{H}_{\beta}(T)}&\P_{\Sigma,\psi}\Big(\|\hat{\Sigma}_{\tau}^{H}-\Sigma\|\ge \bar C_{1}\bar r_{n,p}\Big)\le \bar C_{0}p^{2-\gamma^{2}}\qquad\text{with}\notag\\
  &\qquad\bar r_{n,p}:=S^{1/2}\Big(R^{1-\beta/2}T\Big(\log\frac{n}{\log (ep)}\Big)^{-1+\beta/2}\Big)^{1-q/2}\label{eq:rate}
  \end{align}
  for some numerical constants $\bar C_{0},\bar C_{1}>0$.
\end{thm}
\begin{proof}
It follows from the assumption on $U_*$ that $3TU_*^{\beta}\le RU_*^2$. This and the definition of $U_*$ imply that $8\gamma\sqrt{(\log (ep))/n}\le e^{-RU_*^{2}-3TU_*^\beta}$. Therefore, we can apply Theorem~\ref{thm:hardThreshold}, which yields the result since 
 \begin{align*}
\tau(U_*) &\le 6\gamma\frac{e^{2RU_*^{2}}}{U_*^2}\Big(\frac{\log (ep)}{n}\Big)^{1/2}+3TU_*^{-2+\beta}
\le \Big(\frac{2\bar c}{3} + 3\Big)TU_*^{-2+\beta}.
 \end{align*}
\end{proof}

It is interesting to compare Theorem~\ref{cor:no} with the result of  \citet{ButuceaMatias2005} corresponding to $p=1, S=1,$ and establishing a logarithmic rate  for estimation of the variance in deconvolution model under exponential decay of the Fourier transform of $\epsilon_j$. \citet{ButuceaMatias2005}  have shown that, if $\log |\psi(u)|=\mathcal O(|u|^\beta)$, their estimator achieves asymptotically a mean squared error of the order $(\log n)^{-1+\beta/2}$. This coincides with the case $p=1$ and $q=0$ of the non-asymptotic bound in \eqref{eq:rate}. A similar rate for $p=1$ has been obtained by \citet{matias2002} under the assumptions on the decay of the Laplace transform.

We now turn to the lower bound matching \eqref{eq:rate} for $q=0$. Intuitively, the slow rate comes from the fact that the error distribution can mimic the Gaussian distribution up to some frequency in the Fourier domain. A rigorous application of this observation to the construction of lower bounds goes back to \citet{jacodReiss2014}, though in quite a different setting. For the multidimensional case that we consider here the issue becomes particularly challenging. 


\begin{thm}\label{thm:lowerBound}
  Let $\beta\in(0,2)$ and assume that $C_1p\le S\le C_2 p$, $T(\log n)^{-1+\beta/2}\le C_{3}R^{\beta/2}$, $T(\log n)^{c_*}\ge 1\vee R^{\beta/2}$ for some constants $C_1,C_2,C_3>0$, and $c_*>0$.
  Then, there are constants $c_{1},c_{2}>0$ such that
  \begin{align*}
  \inf_{\tilde{\Sigma}}\sup_{(\Sigma,\psi)\in\mathcal{G}_{0}(S,R)\times\mathcal{H}_{\beta}(T)}\P_{\Sigma,\psi}\big(\|\tilde{\Sigma}-\Sigma\|\ge c_1\underline r_{n,p}\Big) & >c_{2}\quad\mbox{with }\\
  \underline r_{n,p} & :=S^{1/2}R^{1-\beta/2}T\big(\log n\big)^{-1+\beta/2}
  \end{align*}
  where the infimum is taken over all estimators $\tilde{\Sigma}$. 
\end{thm}
The proof of this theorem is postponed to Section~\ref{sec:ProofLowerBound}. We use the method of reduction to testing of many hypotheses relying on a control of the $\chi^2$-divergence between the corresponding distributions, cf. Theorem~2.6 in \cite{tsybakov2009}. The present high-dimensional setting introduces some additional difficulties. When the dimension $p$ of the sample space is growing, an increasing number of derivatives of the characteristic functions has to be taken into account for the $\chi^2$-bound. Achieving  bounds of the correct order in $p$ 
causes difficulty when $p$ is arbitrarily large. We have circumvented this problem by introducing a block structure to define the hypotheses. The construction of the family of covariance matrices of $X_j$ used in the lower bounds relies on ideas from \citet{rigolletTsybakov2012}, while the error distributions are chosen as perturbed $\beta$-stable distributions. To bound the $\chi^2$-divergence, we need a lower bound on the probability density of $Y_j$. It is shown by \citet{butuceaTsybakov2008} that the tails of the density of a one-dimensional stable distribution are polynomially decreasing. We generalize this result to the multivariate case (cf. Lemma~\ref{lem:lowBoundDensity} below) using properties of infinitely divisible distributions.

We now give some comments on the lower bound of Theorem~\ref{thm:lowerBound}. Assuming $S$ of order $p$  means that we consider quite a sparse regime. We always have $S\le p^2$. Recall also that $S\ge p$ as the diagonal of the covariance matrix is included in the definition of $S$ for the class $\mathcal G_0(S,R)$. An alternative strategy pursued in the literature is to estimate a correlation matrix, i.e., to assume that all diagonal entries are known and equal to one.  However, this seems not very natural in the present noisy observation scheme.  On the other hand, Theorem~\ref{thm:lowerBound} shows that even in the sparse regime $S=\mathcal O(p)$ the estimation error tends to $\infty$ as $n\to \infty$ for dimensions $p$ growing polynomially in $n$. The logarithmic in $n$ rate reflects the fact that the present semiparametric problem is severely ill-posed. 

Comparing the lower bound $\underline r_{n,p}$ with the upper bound $\bar r_{n,p}$ from Theorem~\ref{cor:no}, we see that they coincide  if the dimension satisfies $p=\mathcal O(\exp(c n^\gamma))$ for some $\gamma\in[0,1)$ and some $c>0$. Thus, we have established the minimax optimal rate under this condition. Note also that we only loose a factor of order $\log\log p$ for very large $p$, for instance, if $p=e^{n/\log n}$.

\section{Discussion and extensions}\label{sec:ext}

\subsection{The adaptivity issue}

Since the threshold $\tau(U_*)$ in Theorem~\ref{cor:no} depends on unknown parameters $R, T$, and $\beta$, a natural question is whether it is possible  to construct an adaptive procedure independent of these parameters that achieves the same rate. One possibility to explore consists in selecting $\tau$ in a data-driven way. Another option would be to construct estimators corresponding to values of $R, T$, and $\beta$ on a grid, and then to aggregate them.

For direct observations an adaptive choice of the threshold, more precisely a cross-validation criterion, has been proposed by \citet{bickelLevina2008threshold} and was further investigated by \citet{CaiLiu2011}. For noisy observations that we consider here, the adaptation problem turns out to be more delicate since not only an optimal constant has to be selected but also the order of magnitude of $\tau(U)$ depends on the unknown parameter $\beta$.

Often an upper bound $R$ on the maximal entry of $\Sigma$ is known, so that one does not need considering adaptation to $R$. Ignoring the issue of unknown $R$, the choice of the spectral radius $U_*$ of the order $\sqrt{R^{-1}\log(n/\log(e p))}$ is universal, which reflects the fact that the estimation problem is severely ill-posed with dominating bias. Indeed, $U_*$ in Theorem~\ref{cor:no} corresponds to undersmoothing such that the deterministic estimation error dominates the stochastic error without deteriorating the convergence rates. To construct an adaptive counterpart of $\tau$, we need either an estimator of the error of an optimal procedure for estimating $\Sigma$ under the $|\cdot|_\infty$-loss or an estimator of the ``regularity'' $\beta$. Therefore,  extrapolating the argument of \citet{Low1997} to our setting, it seems plausible that an adaptive choice of $\tau$ cannot, in general, lead to the optimal rate. This does not exclude that optimal adaptive estimators can be constructed by other type of procedure, such as aggregation of estimators on the grid as mentioned above.

\subsection{Low-rank covariance matrix}

Alternatively to the above setting where the covariance matrix $\Sigma$ is sparse, we can consider a low-rank matrix $\Sigma$. 
This is of particular interest in the context of factor models where, as discussed by \citet{fanEtAl2011,fanEtAl2013}, an additional observation error should be taken into account. While \cite{fanEtAl2011,fanEtAl2013} estimate the covariance matrix of the noisy observations assuming that the errors have a sparse covariance structure, a spectral approach analogous to the one developed above allows for estimating directly the low-rank covariance matrix of $X$ without sparsity restrictions on the error distribution. 

Such an approach, which is at first sight quite natural, would be to use the spectral covariance estimator $\hat\Sigma$ from \eqref{eq:SigmaHat} together with a nuclear norm penalization. The following oracle inequality is an easy consequence of Theorem~1 in \citet{koltchinskii2011nuclear}.
\begin{prop}\label{prop:oracleNuclearNorm}
  Assume that $\mathbb M\subset\R^{p\times p}$ is convex and let $\tau>0$. On the event $\{2\|\hat\Sigma-\Sigma\|_\infty\le \tau\}$, the estimator $\hat\Sigma^{R}_\tau:=\argmin_{S\in \mathbb M}\{\|S-\hat\Sigma\|^2+\tau \|S\|_1\}$ satisfies
  \[
    \|\hat\Sigma^R_\tau-\Sigma\|^2\le\inf_{S\in\mathbb M}\Big\{\|S-\Sigma\|^2+\big(\frac{1+\sqrt 2}{2}\big)^2\tau^2\rank (S)\Big\}.
  \]
\end{prop}
To use this proposition, we need to find a bound on the spectral norm $\|\hat\Sigma-\Sigma\|_\infty$ that hold with high probability. The techniques from \citet{CaiEtAl2010} designed for the case of direct observations allow us to obtain an upper bound on this quantity of order $p$ up to a logarithmic in $n/\log(p)$ factor. Thus, the convergence rate of this estimator is rather slow. 
%
%
%

Let us show now that another estimator can be constructed based the approach from \citet{belomestnyTrabs2015}, which allows for a better dependence on $p$.
To this end, we write 
\[
-\frac{\langle u,\Sigma u\rangle}{|u|^{2}}=\langle\Theta(u),\Sigma\rangle\quad\mbox{with design matrix }\Theta(u):=-\frac{uu^{\top}}{|u|^{2}},\quad u\in\R^{p}\setminus\{0\}.
\]
For a weight function $w\colon\R^p\to\R_+$ supported on the annulus $\{u\in\R^{p}:\frac{1}{4}\le|u|\le\frac{1}{2}\}$ and a spectral radius $U\ge1$, we set $w_{U}(u):=U^{-p}w(u/U),u\in\R^{p}$.
Motivated by \eqref{eq:regression}, we define the weighted Lasso-type estimator
\begin{equation}
\tilde{\Sigma}_{\lambda}:=\argmin_{M\in\mathbb{M}}\Big\{\int_{\R^{p}}\Big(\frac{\Re\log\phi_{n}(u)\1_{\{|\phi_n(u)|\ge\iota\}}}{|u|^{2}}-\langle\Theta(u),M\rangle\Big)^{2}w_{U}(u)\d u+\lambda\|M\|_{1}\Big\}\label{eq:est}
\end{equation}
for a convex set $\mathbb{M}\subset\{M\in\R^{p\times p}:M\ge0\}$ and with nuclear norm penalisation for some $\lambda>0$. We have inserted a truncation function $\1_{\{|\phi_n(u)|\ge\iota\}}$ for some threshold $\iota>0$ which increases the stability of the estimator by cutting off frequencies with too small point estimates $\phi_n(u)$. Under the universal choice $\iota=1/(2\sqrt n)$ this indicator function will be one with high probability. The estimator $\tilde\Sigma_\lambda$ is associated to the weighted scalar product which replaces the classical empirical scalar product:
\[
\langle A,B\rangle_{U}:=\int_{\R^{d}}\langle\Theta(u),A\rangle\langle\Theta(u),B\rangle w_{U}(u)\d u\quad\mbox{and}\quad\|A\|_{U}^{2}:=\langle A,A\rangle_{U},
\]
for matrices $A,B\in\R^{p\times p}$.  As in \citep[Lemma 3.2]{belomestnyTrabs2015} we have for any for any
positive semi-definite matrix $A\in\R^{p\times p}$ an isometry with respect to the Frobenius norm
\[
\underline\varkappa_{w}\|A\|^{2}\le\|A\|_{U}^{2}\le\overline\varkappa_w\|A\|^2\quad\text{with}\quad\underline\varkappa_{w}:=\int_{\R^{p}}\frac{|v_1|^4}{|v|^4}w(v)\d v,\quad\overline\varkappa_w:=\|w\|_{L^1}.
\]
Adapting slightly the proof of Theorem~1 in \cite{koltchinskii2011nuclear}, we obtain the following
oracle inequality.
\begin{thm}\label{thm:oracleLasso}
Let $\mathbb M$ be convex. Define
\[
\mathcal{R}_{n}:=\int_{\R^{p}}\Big(\frac{\Re\log\phi_{n}(u)\1_{\{|\phi_n(u)|\ge\iota\}}}{|u|^{2}}-\langle\Theta(u),\Sigma\rangle\Big)\Theta(u)w_{U}(u)\d u.
\]
The estimator $\tilde{\Sigma}_{\lambda}$ from \eqref{eq:est} satisfies on the event $\{\|\mathcal{R}_{n}\|_{\infty}\le\lambda\}$ 
\[
\|\tilde{\Sigma}_{\lambda}-\Sigma\|_{U}^{2}\le\inf_{M\in\mathbb{M}}\big\{\|M-\Sigma\|_{U}^{2}+C_*^2\lambda^{2}\rank(M)\big\}
\]
for the constant $C_{*}=(1+\sqrt 2)/(2\underline\varkappa_w)$ depending only $w$.
\end{thm}
We omit the proof of this theorem as it is analogous to Theorem 3.4 in \cite{belomestnyTrabs2015}. In combination with the isometry property we obtain an oracle inequality with respect to the Frobenius norm:
\[
\|\tilde{\Sigma}_{\lambda}-\Sigma\|^{2}\le\inf_{M\in\mathbb{M}}\big\{C_1^*\|M-\Sigma\|^{2}+C_2^*\lambda^{2}\rank(M)\big\}
\]
with $C_1^*=\overline\varkappa_w/\underline\varkappa_w$ and $C_2^*=(1+\sqrt 2)^2/(4\underline\varkappa_w^3)$. The best leading constant in this oracle inequality can be obtained by minimizing $C_1^*$ with respect to $w$. We do not detail it here.

To apply Theorem~\ref{thm:oracleLasso}, we need a sharp probabilistic bound for $\|\mathcal{R}_{n}\|_{\infty}$. At first sight, this might look similar to bounding $\|\hat\Sigma-\Sigma\|_\infty$  in Proposition~\ref{prop:oracleNuclearNorm}. However, the dependence on the dimension is much better because the design matrix satisfies $\|\Theta(u)\|_\infty=1$. 

Consider the error distributions in the subclass of $\mathcal{H}_{\beta}(T)$ defined as follows:
\begin{align*}
 \mathcal{H}_{\beta}'(T) :=\Big\{\psi\text{ characteristic function}:\, &\big|\log |\psi(u)| \big|\le {T}\big(1+|u|_2^{\beta}\big),\, u\in\R^{p}\Big\}\subset \mathcal{H}_{\beta}(T).
\end{align*}

\begin{thm}\label{thm:concentration} 
  Let $T>0$, $\beta\in[0,2)$ and $\psi\in\mathcal H_{\beta}'(T)$ and choose $\iota=\frac{1}{2\sqrt n}$. Then there are constants $C_i=C_i(w)>0,i=1,2,$ depending only on $w$, such that for any $\gamma\ge1$ and any $U\ge 1$ satisfying
  $e^{\|\Sigma\|_\infty U^2/8+2TU^{\beta}} \le \sqrt{n}$ we have $\P(\|\mathcal R_n\|_\infty\ge\lambda)\le3e^{-\gamma^2}$ if
  \begin{equation}\label{eq:lambda}
  \lambda\ge C_1\gamma^2 \frac{e^{\|\Sigma\|_\infty U^2/4+4TU^{\beta}}}{U^2\sqrt n}+C_2TU^{-2+\beta}.
  \end{equation}
\end{thm}

The proof is given in the appendix. The right-hand side of \eqref{eq:lambda} is similar to the threshold \eqref{eq:tau}, but without $\sqrt{\log p}$. Hence, this upper bounds depends on the dimension $p$ only via spectral norm $\|\Sigma\|_\infty$. 
In the well-specified case, $\Sigma\in\mathbb M$ and optimizing over the spectral radius yields $U$ of the order $\sqrt{(\log n)/\|\Sigma\|_\infty}$ and the corresponding $\lambda$ of the order $(\|\Sigma\|_\infty^{-1}\log n)^{-1+\beta/2}$.  The error bound takes the form
\[
\|\tilde{\Sigma}_{\lambda}-\Sigma\|\le C\sqrt{\rank(\Sigma)}\,\|\Sigma\|_\infty^{1-\beta/2}(\log n)^{-1+\beta/2}
\]
with high probability. Here, $C>0$ is a constant depending only on $w$ and $T$. Note that this bound for the estimation error improves a corresponding result in \cite{belomestnyTrabs2015}. In the direct observation case, we can choose $U=\|\Sigma\|_\infty^{-1/2}$ and obtain $\|\tilde{\Sigma}_{\lambda}-\Sigma\|\le C\|\Sigma\|_\infty\sqrt{\rank(\Sigma)/n}$ with high probability.


\subsection{Elliptical distributions}\label{sec:elliptic}

Most of the literature on high-dimensional covariance estimation 
relies on a sub-Gaussian assumption on the distribution of $X_j$. To relax the moment assumption and allow for heavy-tailed distributions, the rich class of elliptical distributions has been studied, see the review paper by \citet{FanEtAl2016}. We refer to \citet{fang1990symmetric} for an introduction to the theory of elliptical distributions. 

We will now outline how our approach can be generalized to the case where $X_j$ follow a centered elliptical distribution, that is the characteristic function of $X_j$ is of the form 
$$\E[e^{i\langle u,X_j\rangle}]=\Phi(u^{\top}\Sigma u),\qquad u\in\R^p,$$
for some scalar function $\Phi\colon\R\to\R$ and some positive definite
matrix $\Sigma$, which is proportional to the covariance matrix.
 The function $\Phi$ is called the
\emph{characteristic generator}. It is easy to see that $\E[X_jX_j^\top]=-2\Phi'(0) \Sigma$ provided that $\Phi$ is differentiable. We impose the mild assumption that $\Phi(\cdot)=\exp(-\eta(\cdot))$ for some function $\eta\colon\R_{+}\to\R_{+}$. Then, the characteristic function of the observations $Y_{j}$ has the form
\[
\phi(u)=\exp\big(-\eta(u^{\top}\Sigma u)+\log\psi(u)\big),\quad u\in\R^{p}.
\]
We recover the Gaussian case with $\eta(x)=\frac{x}{2}$. Other important examples are multivariate $\alpha$-stable distributions where $\eta(x)=x^{\alpha/2}$ for $\alpha\in(0,2]$ or normal mixtures. To adapt
the estimation strategy from Section~\ref{sec:spectralEst}, we assume that $|\Re\log\psi(u)|$ decays slower than $\eta(u^{\top}\Sigma u)$. If $\eta$ is differentiable and strictly monotone with inverse function $\eta^{-1}$, a first order Taylor approximation and the fact that $(\eta^{-1})'=1/(\eta'\circ\eta^{-1})$ yield 
\begin{align*}
\eta^{-1}\big(-\log|\phi(u)|\big) & =\eta^{-1}\big(\eta\big(u^{\top}\Sigma u\big)-\log|\psi(u)|\big)
  \approx u^{\top}\Sigma u-\frac{\log|\psi(u)|}{\eta'(u^{\top}\Sigma u)}.
\end{align*}
If the last term is of smaller order than $u^{\top}\Sigma u=\langle u,\Sigma u\rangle$ for $|u|\to\infty$, we can use these heuristics to estimate $\Sigma$. The argument is made rigorous by the following lemma proved  in the appendix.
\begin{lem}\label{lem:taylorEta}
  Let $\E[e^{i\langle u,X_j\rangle}]=\exp(-\eta(u^{\top}\Sigma u))$ for a positive-definite matrix $\Sigma$ and a strictly monotone function $\eta\colon\R_{+}\to\R_{+}$ which is twice continuously differentiable outside a neighbourhood of the origin. Assume further that
  \[
    \frac{\big|\log|\psi(u)|\big|}{\eta'(\langle u,\Sigma u\rangle)}\le T(1+{|u|})^{\beta}\qquad\text{and}\qquad |x\eta''(x)|\le T|\eta'(x)|,\qquad\text{for all }u\in\R^p,x\in\R_+,
  \]
  for some $\beta<2$ and $T>0$. For all $u\in\R^p$ with $|u|\ge (2^{\beta+1}T^2/\lambda_{min})^{1/(2-\beta)}\vee 1$ we then have
  \[
    \Big|\eta^{-1}\big(-\log|\phi(u)|\big)-\langle u,\Sigma u\rangle-\frac{\log|\psi(u)|}{\eta'(\langle u,\Sigma u\rangle)}\Big|\le \frac{4T^2}{\lambda_{min}}|u|^{2\beta-2},
  \]
  where $\lambda_{min}>0$ is the smallest eigenvalue of $\Sigma$. 
\end{lem}
A major consequence of this lemma for our purposes is that $|u|^{-2}\eta^{-1}\big(-\log|\phi(u)|\big)=\frac{\langle u,\Sigma u\rangle}{|u|^2}+\mathcal O(|u|^{-2+\beta})$ as $|u|\to \infty$. Thus, we can 
act as in Section~\ref{sec:spectralEst}. This leads to the estimator $\hat{\Sigma}^{\Phi} =(\hat{\sigma}_{i,j}^{\Phi})_{i,j=1,\dots p}$ for $\Sigma$ where
\begin{align*}
\hat{\sigma}_{i,i}^{\Phi} & :=\frac{1}{U^{2}}\eta^{-1}\big(-\Re\big(\log\phi_{n}(Uu^{(i)})\big)\big),\\
\hat{\sigma}_{i,j}^{\Phi} & :=\frac{1}{U^{2}}\eta^{-1}\big(-\Re\big(\log\phi_{n}(Uu^{(i,j)})\big)\big)-\frac{\hat{\sigma}_{i,i}^{\Phi}+\hat{\sigma}_{j,j}^{\Phi}}{2}\quad\text{for}\quad i\neq j.
\end{align*}
Applying an argument as in Lemma \ref{lem:taylorEta} together with the linearization for $\log \phi_n$, we can bound the stochastic error of the estimators $\hat\sigma^\Phi_{i,j}$. We obtain the following proposition analogous  to Theorem~\ref{thm:ConcentrationSigma*}. The proof is again postponed to the appendix.
\begin{prop}\label{prop:ConcentrationSigmaPhi} 
  Let the assumptions of Lemma~\ref{lem:taylorEta} be satisfied. Let $\gamma>\sqrt{2}$  and suppose that $U\ge(2^{2+\beta}T^2/\lambda_{min})^{1/(2-\beta)}\vee 1$ satisfies $8\gamma\sqrt{(\log (ep))/n}<\Delta_{\Sigma,U}$ for 
  $$\Delta_{\Sigma,U}:=\min_{i,j} \eta'(U^2\langle u^{(i,j)},\Sigma u^{(i,j)}\rangle)|\phi(Uu^{(i,j)})|.$$ 
  Set  
  \begin{equation*}
    \tau(U)=\frac{12\gamma}{U^2 \Delta_{\Sigma,U}}\sqrt{\frac{\log (ep)}{n}}+4(T+1)U^{-2+\beta}.
  \end{equation*}
  Then, for $c_*= 12e^{-\gamma^{2}}$,
  \[
  \P_{\Sigma,\psi}\Big(\max_{i,j=1,\dots,p}|\hat{\sigma}^\Phi_{i,j}-\sigma_{i,j}|<{\tau(U)}\Big)\ge1- c_*p^{2-\gamma^{2}}.
  \]
\end{prop}
Under more specific assumptions on $\eta$ it is possible to derive a uniform bound for $\Delta_{\Sigma,U}$. Since $|\phi(u)|\ge\exp(-c\Re\eta(u^{\top}\Sigma u))$ for some constant $c>0$, the stochastic error may not explode as fast as for normal
distributions resulting in possibly faster convergence rates depending on $\eta$. Relying on $\hat\Sigma^\Phi$, hard and soft thresholding estimators can be constructed with similar behaviour as for the Gaussian case. 

For the estimator $\hat\Sigma^\Phi$, the function $\eta$ is assumed to be known.  It would be interesting to extend the approach of this section to the case where $\eta$ belongs to a parametric family introducing an additional nuisance parameter. 

\section{Numerical example}\label{sec:sim}
In this section we numerically analyse  the performance of the soft thresholding estimator for the convolution model $Y=X+\varepsilon$, where \(X\) follows a \(p\)-dimensional normal distribution with zero mean and covariance matrix \(\Sigma\) and \(\varepsilon\) is independent of \(X\) and has an elliptical distribution. Specifically, we study the model
\begin{equation*}
\varepsilon\stackrel{d}{=}\sqrt{W}AZ,
\end{equation*}
where \(Z\sim\mathcal{N}(0,I_p)\) has a standard \(p\)-dimensional normal distribution, \(A\) is a \(p\times p\) matrix and \(W\) is a nonnegative random variable with a Laplace transform \(\mathcal{L}.\) As can be easily seen, the characteristic function of \(\varepsilon\) is given by 
\begin{equation*}
\psi(u)=\E\big[e^{i\langle u,\epsilon\rangle}\big]=\mathcal{L}\left(\frac{u^\top AA^\top   u}{2}\right).
\end{equation*}
Thus \(\varepsilon\) has indeed an elliptical distribution. We assume that \(W\) follows a Gamma distribution  with the density
$
p_W(x)=\Gamma(\theta)^{-1}x^{\theta-1}e^{-x},x\geq 0,
$
for some \(\theta>0\). Then we have
\begin{eqnarray*}
\psi(u)=\left(1+\frac{u^\top AA^\top   u}{2}\right)^{-\theta}.
\end{eqnarray*}

Our aim is to compare several estimators of the covariance matrix \(\Sigma\) based on $n$ independent copies \(Y_1,\ldots,Y_n\) of \(Y\). In the direct observations case where $\epsilon=0$ we may apply the sample covariance matrix 
\begin{equation}\label{eq:sampleCov}
\Sigma^{\mathrm{cov}}:=\Sigma_Y^*=\frac{1}{n}\sum_{j=1}^{n}Y_{j}Y_{j}^{\top}.
\end{equation}
Adapting to sparsity in a high-dimensional framework, a soft thresholding estimator based on \(\Sigma^{\mathrm{cov}}\) is given by the solution of the optimisation problem, cf. \citet{rothmanEtAl2009},
\begin{equation}
\Sigma_{\tau}^{\mathrm{s}}:=\argmin_{S\in\R^{p\times p}}\big\{|S-\Sigma^{\mathrm{cov}}|_{2}^{2}+2\tau|S|_{1}\big\},\label{eq:soft_thr}
\end{equation}
with threshold parameter $\tau>0$. In some situations positive definiteness of the covariance matrix estimate is desirable when
the covariance estimator is, for example, applied to  supervised learning or if one needs to generate samples from the underlying normal distribution. In order to achieve 
positive definiteness,  \citet{rothman2012positive} proposed to use the following modification of \eqref{eq:soft_thr}: 
\begin{equation}
\Sigma_{\tau}^{\mathrm{pds}}:=\argmin_{S\in\R^{p\times p},\,S\succ 0}\big\{|S-\Sigma^{\mathrm{cov}}|_{2}^{2}+2\tau|S|_{1}-\lambda \log |S|\big\},\label{eq:pdsoft}
\end{equation}
where \(|S|\) denotes the determinant of the matrix \(S\) and \(\lambda\) is a fixed small number. The logarithmic barrier term in \eqref{eq:pdsoft} ensures
the existence of a positive definite solution,  since \( \log |S|= \sum_{j=1}^p \log (\sigma_j(S))\), where \(\sigma_j(S)\) is the \(j\)th largest eigenvalue of \(S> 0.\) In order to solve \eqref{eq:pdsoft}, an algorithm similar to the
graphical lasso algorithm can be applied, see \citet{friedman2008sparse}. 

Turning back to our deconvolution problem, we have already seen that the estimators  \eqref{eq:sampleCov}, \eqref{eq:soft_thr} and \eqref{eq:pdsoft}  fail to deliver a consistent estimator for \(\Sigma\)  unless \(\varepsilon\) is zero. Hence, we finally introduce the  positivity preserving version of the spectral soft thresholding estimator from \eqref{eq:soft}:
\begin{equation}
\Sigma_{\tau}^{\mathrm{sps}}:=\argmin_{S\in\R^{p\times p},\,S\succ 0}\big\{|S-\hat{\Sigma}|_{2}^{2}+2\tau|S|_{1}-\lambda \log |S|\big\}.\label{eq:spdsoft}
\end{equation}
The  tuning parameter \(\tau\) can be chosen using  a method introduced in \cite{bickelLevina2008threshold}. The data is randomly  partitioned  \(N\) times into a training set of size \(n_1\) and a validation set of size \(n_2\) with \(n_2=\lfloor n/\log(n)\rfloor\) and \(n_1=n-n_2.\) The  tuning parameter is 
then selected as \(\widehat\tau=\argmin_{\tau}\mathcal{Q}_N(\tau),\) where 
\begin{eqnarray*}
\mathcal{Q}_N(\tau)=\sum_{m=1}^N \|\Sigma_{\tau}^{\mathrm{sps},(m,n_1)}-\hat{\Sigma}^{(m,n_2)}\|^2,
\end{eqnarray*}
where \(\Sigma_{\tau}^{\mathrm{sps},(m,n_1)}\) is the estimator, with penalty parameter \(\tau,\) computed with the training set of the \(m\)th split and \(\hat{\Sigma}^{(m,n_2)}\) is the estimator \eqref{eq:SigmaHat} computed with the validation set of the \(m\)th split. 

First, we consider a tridiagonal model where the population covariance matrix \(\Sigma\) has entries \(\sigma_{ij} = 0.4\cdot 1(|i-j | = 1) + 1(i = j ),\) \(i,j\in \{1,2,\ldots,p\}.\) Using this covariance model with $p=20$, we generate \(n = 50\) realizations of independent normal random vectors with mean zero and the covariance matrix \(\Sigma\). Adding an independent noise \(\varepsilon\) with the above elliptical distribution with \(A=I_d\), depending on the parameter \(\theta\), we compute three estimates \(\Sigma^{\mathrm{cov}},\) \(\Sigma_{\tau}^{\mathrm{pds}}\) and \(\Sigma_{\tau}^{\mathrm{sps}}\). This procedure was repeated \(500\) times. The parameters of the algorithms are \(\tau=0.25,\) \(\lambda=10^{-4},\) where the parameter \(\tau\) is selected as a minimum of the function \(\mathcal{Q}_{100}(\tau)\)  shown in Figure~\ref{fig: obj_cv}.
  
\begin{figure}[t]
\centering
\includegraphics[width=0.8\linewidth]{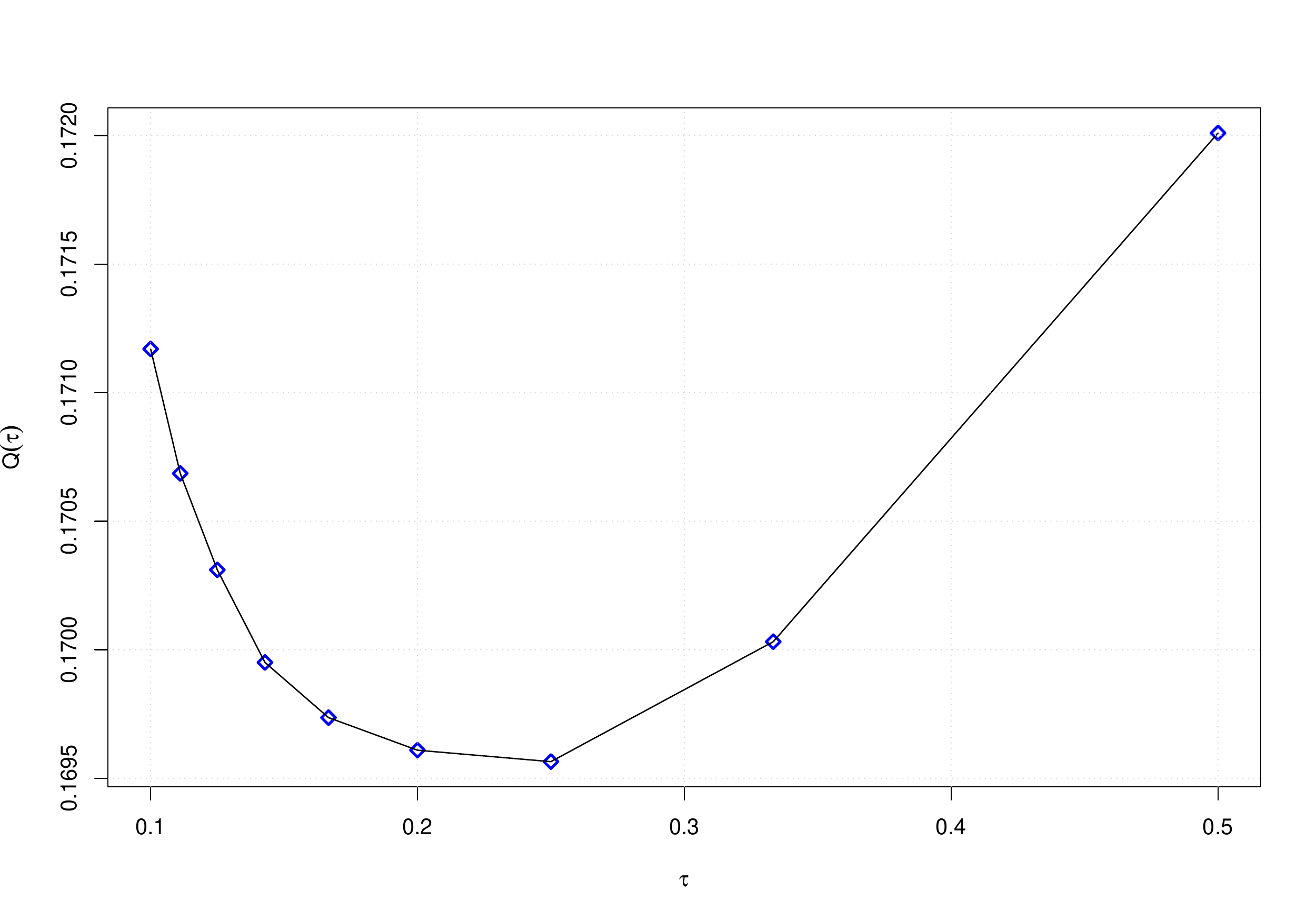}
 \caption{The objective function \(\mathcal{Q}_{100}(\tau)\) for the choice of the tuning parameter \(\tau\)
\label{fig: obj_cv}}
\end{figure}

 The results are presented in Figure~\ref{fig: gamma_errors_1}
for the case of direct observations and for three different noise specifications corresponding to the values \(\theta\in\{0.5,1,2\}.\) The used values of the tuning parameter \(U\) are \(1,3,3,\) respectively.  While in the case of direct observations, the estimator \(\Sigma_{\tau}^{\mathrm{sps}}\) has no advantages over  \(\Sigma^{\mathrm{cov}}\) and \(\Sigma_{\tau}^{\mathrm{pds}},\) it significantly outperforms these two estimators in the case of non-zero noise. We do not only observe a strong bias for \(\Sigma^{\mathrm{cov}}\) and \(\Sigma_{\tau}^{\mathrm{pds}}\) in the presence of noise, but also a much better concentration of the spectral estimator $\Sigma_{\tau}^{\mathrm{sps}}$ compared to the other two procedures. The higher is the variance of the noise, the stronger are these bias and variance effects.

\begin{figure}[t]
\centering
\includegraphics[width=0.8\linewidth]{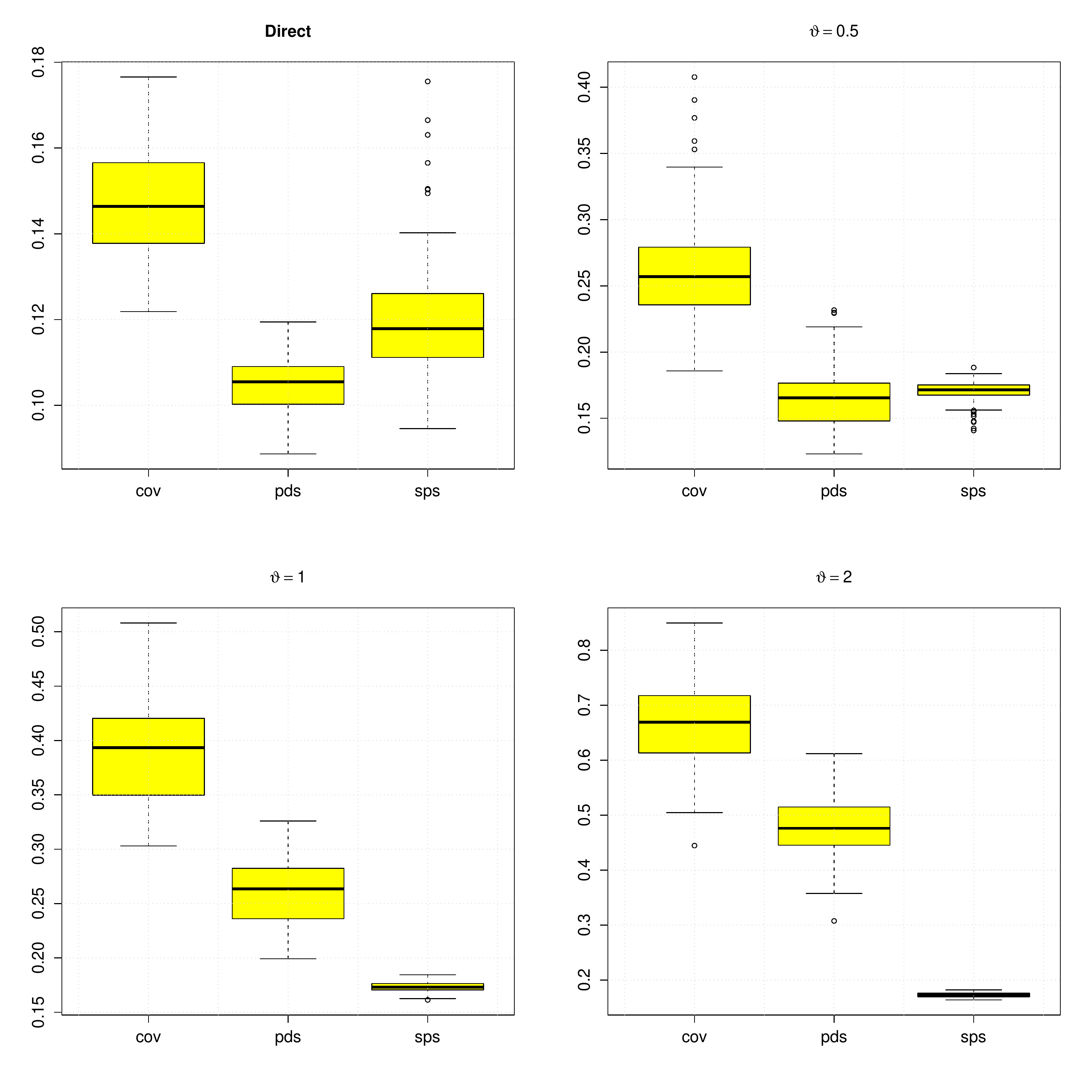}
 \caption{Tridiagonal \(\Sigma:\) box plots of the estimation errors \(\|\Sigma_{\tau}^{o}-\Sigma\|\) for \(o\in\{\mathrm{cov},\mathrm{pds},\mathrm{sps}\}\) in the case of  the convolution model  \(Y=X+\varepsilon\) with \(\varepsilon\stackrel{d}{=}\sqrt{W}Z,\) where \(Z\sim\mathcal{N}_{20}(0,I_{20})\) and \(W\sim \mathrm{Gamma}(\vartheta).\)
\label{fig: gamma_errors_1}}
\end{figure}

\par

Now, let us consider  the case of normal noise. Note that  this situation corresponds to \(\beta=2\) and is not covered (at least formally) by our theoretical study.  Specifically we generate samples from the model $Y=X+\varepsilon$, where \(X\) follows a \(p\)-dimensional normal distribution with zero mean and covariance matrix \(\Sigma\) and \(\varepsilon\) is independent of \(X\) and has also normal distribution with zero mean and covariance matrix \(\rho^2 I\).  We again consider tridiagonal model where the population covariance matrix \(\Sigma\) has entries \(\sigma_{ij} = 0.4\cdot 1(|i-j | = 1) + 1(i = j ),\) \(i,j\in \{1,2,\ldots,p\}.\) 
In Figure~\ref{fig: norm_errors} the corresponding estimation errors for three methods are presented in the case of \(p=20,\) \(\tau=0.4,\) \(n=50\) and \(\rho\in \{0.1,0.5\}.\) As one can see, even in the case of misspecified models the spectral estimator continues to perform reasonably well. 
\begin{figure}[t]
\centering
\includegraphics[width=0.8\linewidth]{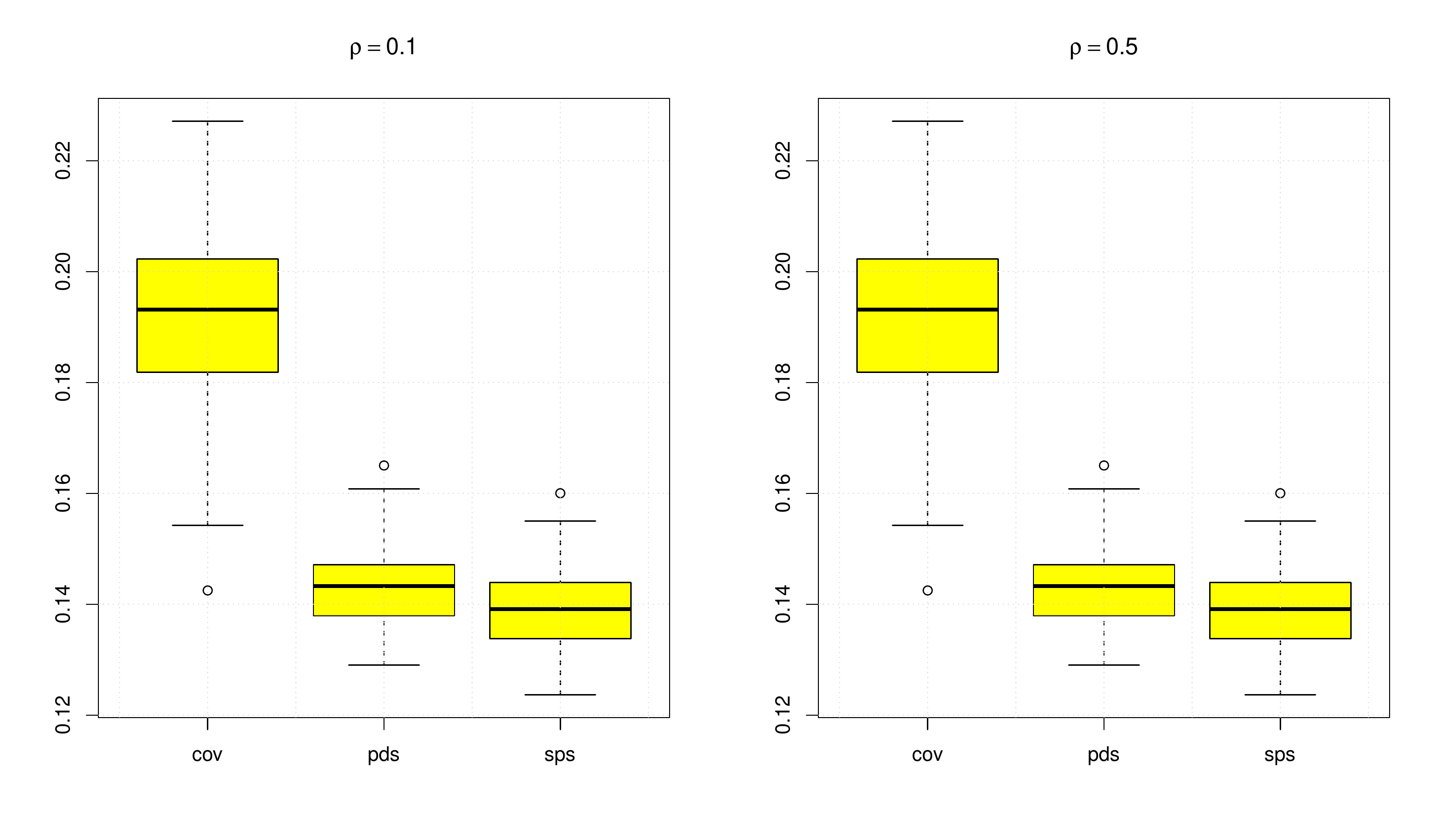}
 \caption{Tridiagonal \(\Sigma:\) box plots of the estimation errors \(\|\Sigma_{\tau}^{o}-\Sigma\|\) for \(o\in\{\mathrm{cov},\mathrm{pds},\mathrm{sps}\}\) in the case of  the convolution model  \(Y=X+\varepsilon\) with \(\varepsilon\stackrel{d}{=} Z,\) where \(Z\sim\mathcal{N}_{20}(0,\rho I_{20}).\) 
\label{fig: norm_errors}}
\end{figure}
\par
Finally, we study the situation where the matrix \(\Sigma\) is block diagonal with the elliptical error distribution from above. In particular, we generate positive definite matrix with randomly-signed, non-zero elements. A shift is added to the diagonal of the matrix so that its condition number equals \(p\).  Using this covariance model, we generated \(n = 100\) realizations of independent \(20\)-dimensional normal random vectors with mean zero and covariance \(\Sigma\). We then proceed as before considering the case of direct observations and \(\theta\in\{0.5,1,2\}\). The tuning parameter \(U\) was taken to be \(3\) for all three cases. The errors show a similar behaviour as in the first case, see Figure~\ref{fig: gamma_errors_2}.
 \begin{figure}[t]
\centering
\includegraphics[width=0.8\linewidth]{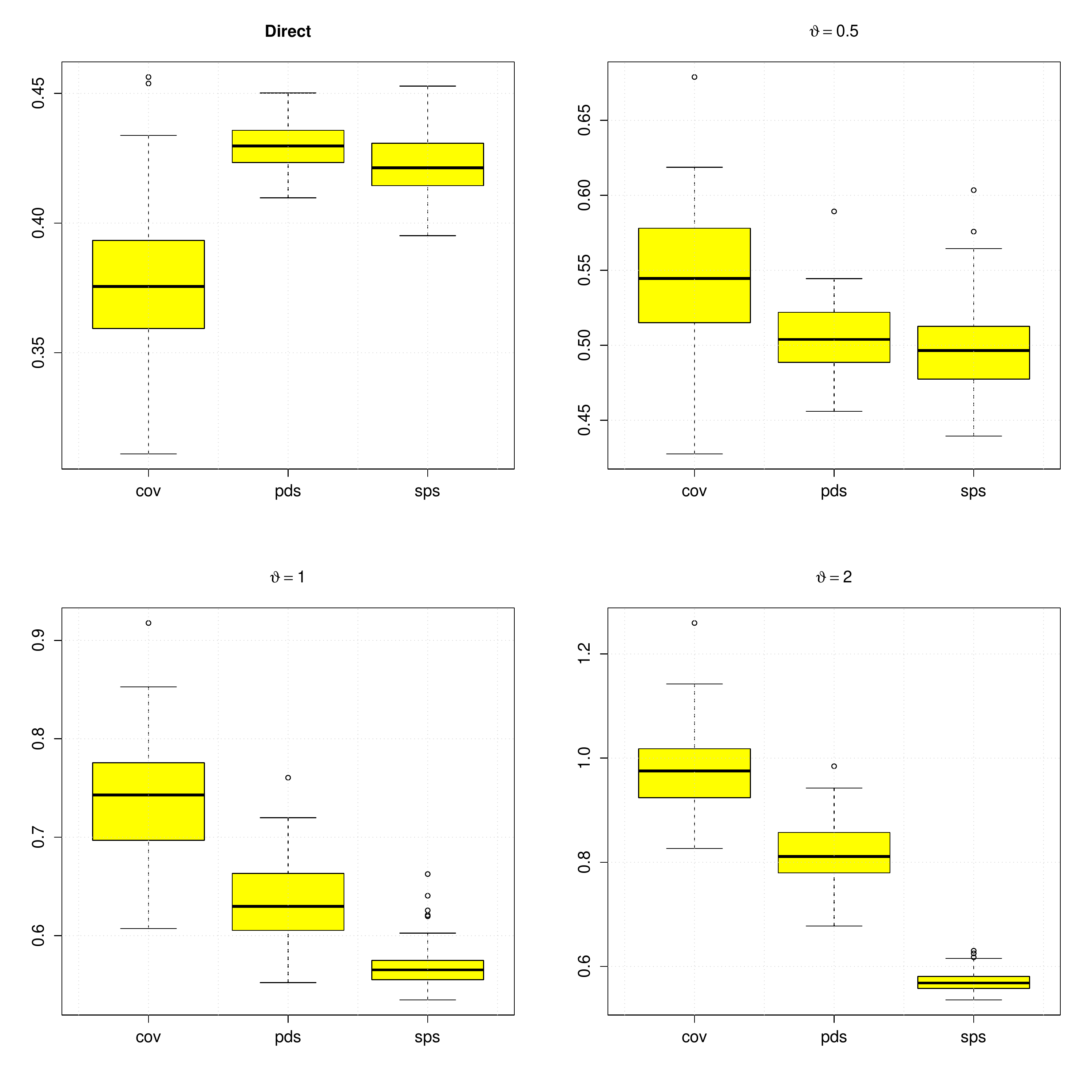}
 \caption{Block diagonal \(\Sigma:\) box plots of the estimation errors \(\|\Sigma_{\tau}^{o}-\Sigma\|\) for \(o\in\{\mathrm{cov},\mathrm{pds},\mathrm{sps}\}\) in the case of  the convolution model  \(Y=X+\varepsilon\) with \(\varepsilon\stackrel{d}{=}\sqrt{W}Z,\) where \(Z\sim\mathcal{N}_{20}(0,I_{20})\) and \(W\sim \mathrm{Gamma}(\vartheta).\)
\label{fig: gamma_errors_2}}
\end{figure}

\section{Proofs}\label{sec:proofs}

\subsection{Concentration of the spectral estimator}\label{sec:proofSigma}
For the proof of Theorem~\ref{thm:ConcentrationSigma*}, we need the following lemmas. Set $S(u)=\Re(\log\phi_{n}(u)-\log\phi(u))$.
\begin{lem}\label{lem:reduction}  
For any $x\in(0,1]$, and any $u\in\R^{p}$ such that $\phi(u)\neq0$,
\[
\P\big(|S(u)|\ge x\big)\le3\P\Big(|\phi_{n}(u)-\phi(u)|\ge\frac{x}{2}|\phi(u)|\Big).
\]
\end{lem}
\begin{proof}
We have 
\begin{equation*}
S(u)=\log \Big|\frac{\phi_{n}(u)}{\phi(u)}\Big| \le \log\Big(\Big|\frac{\phi_{n}(u)-\phi(u)}{\phi(u)}\Big|+1\Big)\le \Big|\frac{\phi_{n}(u)-\phi(u)}{\phi(u)}\Big|.
\end{equation*}
Thus, $\P\big(S(u)\ge x\big)\le \P\big(|\phi_{n}(u)-\phi(u)|\ge x|\phi(u)|\big)$ for all $x>0$. Next, on the event $\Big\{|\phi_{n}(u)-\phi(u)|\le\frac{1}{2}|\phi(u)|\Big\}$ we have 
\begin{equation*}
- S(u)=\log
 \Big|\frac{\phi(u)}{\phi_n(u)}\Big| \le \log\Big(\Big|\frac{\phi_{n}(u)-\phi(u)}{\phi_n(u)}\Big|+1\Big)\le \log\Big(2\Big|\frac{\phi_{n}(u)-\phi(u)}{\phi(u)}\Big|+1\Big)\le 2\Big|\frac{\phi_{n}(u)-\phi(u)}{\phi(u)}\Big|.
\end{equation*}
Therefore, for any $x>0$,
$$
\P\big(-S(u)\ge x\big)\le \P\Big(2|\phi_{n}(u)-\phi(u)|\ge x|\phi(u)|\Big) + \P\Big(|\phi_{n}(u)-\phi(u)|>\frac{1}{2}|\phi(u)|\Big).
$$
Since $x\in (0,1]$, we obtain $\P\big(-S(u)\ge x\big)\le 2\P\big(|\phi_{n}(u)-\phi(u)|\ge (x/2)|\phi(u)|\big)$ and hence the lemma. 
\end{proof}
\begin{lem}
\label{lem:concentrationPhi}For any $\kappa\in(0,\sqrt{n}/8]$ we
have 
\[
\P\Big(|\phi_{n}(u)-\phi(u)|\ge\frac{3\kappa}{\sqrt{n}}\Big)\le4e^{-\kappa^{2}}.
\]
\end{lem}

\begin{proof}
We decompose $\phi_{n}-\phi$ into real and imaginary part. Both can
be estimated analogously, such that we consider only the real part.
We write 
\[
\Re\big(\phi_{n}(u)-\phi(u)\big)=\frac{1}{n}\sum_{k=1}^{n}\xi_{k}(u)\quad\mbox{with}\quad\xi_{k}(u):=\Re\Big(e^{i\langle u,Y_{k}\rangle}\Big)-\Re\phi(u).
\]
The independent and centred random variables $\xi_{k}(u),k=1,\dots,n,$
satisfy 
\begin{align*}
|\xi_{k}(u)|\le & 2\quad\mbox{and}\quad\Var(\xi_{k}(u))\le1-|\phi(u)|^{2}\le1.
\end{align*}
Using the fact that $\kappa\in(0,\sqrt{n}/8]$ and then applying Bernstein's inequality we find
\[
\P\big(\big|\Re\big(\phi_{n}(u)-\phi(u)\big)\big|\ge\frac{3\kappa}{2\sqrt{n}}\big)
\le\P\Big(\big|\Re\big(\phi_{n}(u)-\phi(u)\big)\big|\ge\frac{\sqrt{2}\kappa}{\sqrt{n}}+\frac{2\kappa^{2}}{3n}\Big)\le2e^{-\kappa^{2}}.
\]
\end{proof}
\begin{cor}\label{cor:concStochErr}
  For any $\gamma>0$ and $u\in \R^p$ such that $\gamma\sqrt{(\log (ep))/n}\le |\phi(u)|/8$ we have
  \[
    \P\Big(|S(u)|\ge\frac{6\gamma\sqrt{\log (ep)}}{\sqrt{n}|\phi(u)|}\Big)\le 12(ep)^{-\gamma^{2}}.
  \]
\end{cor}
\begin{proof} We use Lemma~\ref{lem:reduction} with $x=\frac{6\gamma\sqrt{\log (ep)}}{\sqrt{n}|\phi(u)|}$ and then Lemma~\ref{lem:concentrationPhi} with $\kappa= \gamma\sqrt{\log (ep)}$. 
To apply Lemma~\ref{lem:reduction}  we need $6\gamma\sqrt{\frac{\log ep}{n}}\le |\phi(u)|$, while  Lemma~\ref{lem:concentrationPhi} requires $8\gamma\sqrt{\frac{\log ep}{n}}\le 1$. Since $|\phi(u)|\le1$ both conditions are satisfied.
\end{proof}

\section{\label{sec:ProofLowerBound}Proof of the lower bound: Theorem~\ref{thm:lowerBound}}

\noindent Since $C_1 p\le S \le C_2 p$ it is enough to assume that $2p\le S$ (otherwise we consider a $(C_1p/2)$-dimensional subspace). 
Furthermore, we will assume without loss of generality that  $S=p+2k$ for some integer $k\ge1$ corresponding to $p$ non-zero diagonal entries and $2k$ non-zero off-diagonal entries of the covariance matrix. Note that under our assumptions, $S,k$ and $p$ are of the same order up to constants:
\begin{equation} \label{spk}
  \frac S4\le k=\frac{S-p}{2}\le \frac S2\le\frac{C_2p}{2}.
\end{equation}
Let $\P_{\Sigma,\psi}$ denote the distribution of $Y_{j}$ corresponding to the
covariance matrix $\Sigma\in\mathcal{G}_{q}(S,R)$ and to the error distribution
with characteristic function $\psi\in\mathcal{H}_{\beta}(T)$. Set
\[
\phi_{\Sigma,\psi}(u):=\E_{\Sigma,\psi}[e^{i\langle u,Y_{j}\rangle}]=\exp\Big(-\frac{1}{2}\langle u,\Sigma u\rangle+\log\psi(u)\Big).
\]
Applying Theorem~2.6 in \cite{tsybakov2009}, it is sufficient to
construct a finite number of pairs $(\Sigma_{i},\psi_{i})$ with $\Sigma_{0}=RI_{p},\psi_{0}\in\mathcal{H}_{\beta}(T)$
and $(\Sigma_{i},\psi_{i})\in\mathcal{G}_{q}(p+2k,R)\times\mathcal{H}_{\beta}(T)$
for $i=1,\dots,M$, such that the following two conditions hold:
\begin{enumerate}
\item $\|\Sigma_{i}-\Sigma_{j}\|\ge CS^{1/2}T\big(R^{-1}\log n\big)^{-1+\beta/2}$
for all $0\le i< j\le M$ and some constant $C>0$,
\item $\chi^{2}(\P^{\otimes n}_{\Sigma_{j},\psi_{j}},\mathbb{P}^{\otimes n}_{\Sigma_{0},\psi_{0}})\le M/3$ for all $j=1,\dots,M$.
\end{enumerate}
\paragraph*{Step 1: Constructing the pairs $(\Sigma_{i},\psi_{i})$.}
Without loss of generality, consider $p$ that can be decomposed as $p=Lb$ where $b$ and $L$ are integers. For a block size $b\in\N$ and $L=p/b\in\N$ let $\mathcal{B}\subset\R^{p\times p}$
denote the set of symmetric block diagonal matrices $B=\diag(A_1,\dots,A_L)$ satisfying:
\begin{itemize}
 \item $B=(b_{ij})$ has exactly $k$ non-zero over-diagonal entries, all equal to $1$;
 \item $b_{ii}=0$ for $i=1,\dots,n$;
 \item $A_l\in\R^{b\times b}$ for $l=1,\dots,L$.
\end{itemize}
There are $N:=Lb(b-1)/2=p(b-1)/2$ positions over the diagonal of $B$ where the entry 1 can possibly appear. Since $k\le C_2p/2$, we have $k< N$ for $b>C_2+1$. In what follows, we select $b>C_2+1$, which is a fixed integer independent of $k$ and $p$. 
Lemma A.3 in \citet{rigolletTsybakov2011} yields that there is a subset $\{B_{1},\dots B_{M}\}\subset\mathcal{B}$
such that for any $i\neq j$ we have $\|B_{i}-B_{j}\|^{2}\ge(k+1)/4$
and for some constants $C_{1}',c_1'>0$, 
\begin{equation}
\log M\ge C_{1}' k\log\Big(1+\frac{eN}{4k}\Big)
\ge C_{1}' k\log\Big(1+\frac{c_1'bp}{k}\Big).\label{eq:numAlternatives}
\end{equation}
We consider now the following family of matrices
\[
\Sigma_{0}=RI_{p}, \quad \Sigma_{j}=RI_{p}+\frac{\rho T}{b}\delta_{n,p}^{2-\beta}B_{j},\quad j=1,\dots,M,
\]
where 
\begin{equation}\label{eq:deltaN}
  \delta_{n,p}=R^{1/2}\Big(6\log\frac{n}{\rho' \log(1+c_1' bp/k)}\Big)^{-1/2},
\end{equation}
and $\rho, \rho'>0$ are small enough constants to be chosen later.
By construction and using \eqref{spk} we have 
\begin{align*}
  \|\Sigma_{i}-\Sigma_{j}\|&\ge  \rho\frac{T}{2b}\delta_{n,p}^{2-\beta}k^{1/2}
  \ge \frac{T}{2b} k^{1/2}\Big(6R^{-1}\log\Big(\frac{n}{\rho'\log\big(1+\frac{c_1'bp}{k}\big)}\Big)\Big)^{-(1-\beta/2)}\\
&\ge c_2'T S^{1/2}\big(R^{-1}\log n\big)^{-(1-\beta/2)}
\end{align*}
where $c_2'>0$ is a constant. 
Moreover, since by assumption of the theorem, $R^{-1}Tb^{-1}\delta_{n,p}^{2-\beta}$ is uniformly bounded, the matrices $\Sigma_{i}$ are diagonally dominant and thus positive semi-definite for sufficiently small $\rho$. We conclude that the $\Sigma_{i}$ thus defined are covariance matrices satisfying the lower bound in (i) above.

We now turn to the construction of characteristic functions $\psi_j$.
To have an as small as possible $L^{2}$-distance between the characteristic functions,
we choose $\psi_{j}$ such that $\log\psi_{j}(u)-\log\psi_{0}(u)$
mimics $\langle u,(\Sigma_j-\Sigma_0)u\rangle/2$ for
small frequencies, keeping the block structure. In what follows, we denote by $\F$ the Fourier transform operator. On each block of the matrix $B_j=\diag(A_{j,1},\dots,A_{j,L})$,  for $j=1,\dots,M,l=1,\dots,L$, we define
\begin{align*}
\log\psi_{j,l}(u):= & \frac{\rho T\delta_{n,p}^{2-\beta}}{2b}\langle u,A_{j,l}u\rangle\F K(\delta_{n,p}u)+\log\psi_{0,l}(u),\qquad u\in\R^b,\\
\log\psi_{0,l}(u):= & \int_{\R^{b}}\big(e^{i\langle u,x\rangle}-i\langle u,x\rangle\1_{\{\beta\ge1\}}-1\big)\frac{T}{\xi_b|x|^{\beta+b}}\d x, \qquad u\in\R^b,
\end{align*}
where $\xi_b>0$ is a constant depending only on $b$, and $K\in L^{1}(\R^{b})\cap C^{2}(\R^b)$ is a function satisfying
$\F K\in C^{\infty}(\R^b)$, and  
\[
\F K(u)=1\quad\mbox{for}\ |u|\le1, \quad \F K(u)=0\quad\mbox{for}\ |u|>2, \quad\mbox{and}\quad 0\le \F K(u)\le 1\quad \forall \ u.
\]
Writing $u^l:=(u_{b(l-1)+1},\dots,u_{bl})$ for $1\le l\le L$ and $u\in\R^p$, we then set
\[
  \psi_j(u):=\prod_{l=1}^L\psi_{j,l}(u^l),\qquad j=0,\dots,M.
\]
Note that $\psi_{0,l}$ is the characteristic function of a $b$-dimensional symmetric stable distribution, cf. \citet[Thm. 14.3]{sato1999levy}. To check that $\psi_{0}\in\mathcal{H}_{\beta}(T)$ is satisfied, we use Theorem~14.10 in \cite{sato1999levy}, which yields
\begin{align*}
 \big|\log|\psi_{0}(u)| \big|&\le \sum_{l=1}^L \big|\log|\psi_{0,l}(u)| \big|
  \le \sum_{l=1}^L C_\beta \frac{T}{\xi_b}\frac{2\pi^{b/2}}{\Gamma(b/2)}|u^l|^\beta
  \le C_\beta \frac{T}{\xi_b}\frac{2\pi^{b/2}}{\Gamma(b/2)}|u|_\beta^\beta,
\end{align*}
where $C_\beta>0$ is a constant depending only on $\beta$ and where $\frac{2\pi^{b/2}}{\Gamma(b/2)}$ is the surface area of the $(b-1)$-dimensional sphere. Thus, choosing
\begin{equation}
  \xi_b=c\frac{2\pi^{b/2}}{\Gamma(b/2)}\label{eq:xi} 
\end{equation}
for some sufficiently large $c>0$ guarantees that $\psi_{0}\in\mathcal{H}_{\beta}(T)$. Note that $\xi_b$ is bounded uniformly in $b$.

We also have $\psi_{j}\in\mathcal{H}_{\beta}(T)$ for sufficiently small $\rho$ since maximal singular value $\|A_{j,l}\|_\infty\le b$ and
\begin{align}
  \sum_{l=1}^L\Big|\frac{\rho T\delta_{n,p}^{2-\beta}}{2b}\langle u^l,A_{j,l}u^l\rangle\F K(\delta_{n,p}u^l)\Big|
  &\le\frac{\rho\|K\|_{L^1} T}{2b}\delta_{n,p}^{2-\beta}\sum_{l=1}^L\|A_{j,l}\|_\infty|u^l|^{2}\1_{\{|u^l|\le2/\delta_{n,p}\}}\notag\\
  &\le\frac{\rho\|K\|_{L^1} T}{2b}\delta_{n,p}^{2-\beta}\sum_{l=1}^Lb\big(\frac 2{\delta_{n,p}}\big)^{2-\beta}|u^l|^{\beta}\label{eq:problem}\\
  &\le2\rho\|K\|_{L^1} T|u|_\beta^\beta.\notag
\end{align}
It remains to verify that $\psi_{j,l},l=1,\dots,L$ are indeed characteristic functions.
Denoting $A_{j,l}=(a_{k,m}^{j,l})_{k,m=1,\dots,b}$ and 
\[
\nu_{j,l}:=\frac{1}{2b}\sum_{k,m}a_{k,m}^{j,l}(\partial_{k}\partial_{m}K),
\]
where $\partial_{k}$ stands for the derivative with respect to $k$th coordinate, we rewrite the characteristic exponent as
\begin{align*}
\log\psi_{j,l}(u) & =\frac{\rho T\delta_{n,p}^{2-\beta}}{2b}\sum_{k,l=1}^{b}a_{k,m}^{j,l}u_{k}u_{l}\F K(\delta_{n,p}u)+\log\psi_{0}(u)\\
 & =-\F\Big[\rho T\delta_{n,p}^{-\beta-b}\nu_{j,l}(\delta_{n,p}^{-1}\cdot)\Big](u)+\log\psi_{0}(u)\\
 & =\int_{\R^{b}}\big(e^{i\langle u,x\rangle}-i\langle u,x\rangle\1_{\{\beta>1\}}-1\big)\Big(\frac{T}{\xi_b|x|^{\beta+b}}-\rho T\delta_{n,p}^{-\beta-b}\nu_{j,l}(x/\delta_{n,p})\Big)\d x
\end{align*}
where in the last line we have used the relations $\int_{\R^{b}}\nu_{j,l}(x)\d x=\F\nu_{j,l}(0)=0$
and, if $\beta\ge1$, $\int_{\R^{b}}i\langle u,x\rangle\nu_{j,l}(x)\d x=\langle u,\nabla(\F\nu_{j,l})(0)\rangle=0$
for any $u\in\R^{b}$. Consequently, $\psi_{j,l}$ is the characteristic
function of an infinitely divisible distribution with L\'evy density 
$T\xi_b^{-1}|x|^{-\beta-b}-\rho T\delta_{n,p}^{-\beta-b}\nu_{j,l}(x/\delta_{n,p})$
provided that the latter is non-negative. To check this, it is enough to verify the equivalent condition   $\rho\xi_b\nu_{j,l}(x)\le|x|^{-\beta-b}$
for all $x\in\R^{b}\setminus\{0\}$ and some sufficiently small $\rho$. We have
\begin{align}
  \||x|^{\beta+b}\nu_{j,l}(x)\|_\infty &\le \|\nu_{j,l}\|_\infty+\||x|^{2\lceil(\beta+b)/2\rceil}\nu_{j,l}(x)\|_\infty\notag\\
  &\le \|\mathcal F\nu_{j,l}\|_{L^1}+\|\Delta^{\lceil(\beta+b)/2\rceil}\mathcal F\nu_{j,l}\|_{L^1}\label{eq:supNu}
\end{align}
where $\Delta$ denotes the Laplace operator, $\lceil x \rceil$ is the minimal integer greater than $x$, and $\|\cdot\|_{L^q}$ stands for the $L_q(\R^b)$-norm. By construction, $\F\nu_{j,l}(u)=\frac{1}{2b}\langle u,A_{j,l} u\rangle \F K(u)$, and thus
\begin{align*}
  \|\mathcal F\nu_{j,l}\|_{L^1}\le\frac{\|A_{j,l}\|_\infty}{2b}\big\||u|^2\F K(u)\big\|_{L^1}\le\frac{1}{2}\big\||u|^2\F K(u)\big\|_{L^1}
\end{align*}
where we have used the inequality $\|A_{j,l}\|_\infty\le b$. Since the support of $\F K$ is compact the last expression is bounded. The second term in \eqref{eq:supNu} admits an analogous bound.

\paragraph*{Step 2: Bounding the $\chi^2$-divergence.}

Due to the block structure, for any pair $(\Sigma_{i},\psi_{i})$  we have
{$$
\P_{\Sigma_{i},\psi_{i}} =\prod_{l=1}^L \P_{i,l}
$$
for all $i=1,\dots,M$, $l=1,\dots,L$, where $\P_{i,l}$ is the convolution of the normal distribution $\mathcal N(0,RI_b+\frac{\rho T}{b}\delta_{n,p}^{2-\beta}A_{i,l})$ on $\R^b$ with a distribution given by the characteristic function $\psi_{i,l}$. We also denote by $\P_{0}$ the convolution of $\mathcal N(0,RI_b)$ with the stable distribution given by $\psi_{0,l}$.  
}
We have
\begin{align}
  \chi^{2}(\P^{\otimes n}_{\Sigma_{i},\psi_{i}},\P^{\otimes n}_{\Sigma_{0},\psi_{0}})&=\big(1+\chi^2(\P_{\Sigma_{i},\psi_{i}},\P_{\Sigma_{0},\psi_{0}})\big)^n-1\notag\\
  &=\prod_{l=1}^L\big(1+\chi^2(\P_{i,l},\P_{0})\big)^n-1.\label{eq:chi2}
\end{align}
Thus, to check condition (ii) stated at the beginning of this subsection, we need to bound from above the value  \begin{equation}
\chi^2(\P_{i,l},\P_{0})=\int_{f_{0}(x)>0}\frac{\big(f_{i,l}(x)-f_{0}(x)\big)^{2}}{f_{0}(x)}\d x\label{eq:Chi2}
\end{equation}
where $f_{i,l}$ and $f_0$ are the densities of $\P_{i,l}$ and $\P_0$, respectively.
To this end, we first establish a lower bound for $f_{0}$, which
is the density of the convolution of a normal 
distribution on $\R^b$ with zero mean and covariance matrix $RI_b$ and a stable distribution on $\R^b$. {If there is no Gaussian component, we write $R=0$ referring to a convolution of the stable distribution with a Dirac measure in zero.}
\begin{lem}\label{lem:lowBoundDensity}
  In the special case of a standard stable density $f_0$ ($R=0$, $T=1$) and $\beta\in(0,2)$ we have $f_0(x)\ge C_b(1+|x|^{\beta+b})^{-1}$ for a constant $C_b>0$ depending only on $b$. If $R>0$ and $T>0$ are such that $T(\log n)^{-c}\le CR^{\beta/2}$ for some $C,c>0$, we have the lower bound
  \[
    f_0(x)\ge C'_b R^{-b/2}(\log n)^{-cb/\beta}\frac{1}{(1+T^{-1-b/\beta}|x|^{b+\beta})}
  \]
  for another constant $C'_b>0$.
\end{lem}
\begin{proof}
  { \emph{Step 1:} We first consider the case $R=0,T=1$ and start with $\beta\in(0,1)$}. We have $f_0=h_{c}\ast h_{f}$
  where 
  \begin{align*}
  h_{c}(x):=&\F^{-1}\Big[\exp\Big(\frac{1}{\xi_b}\int_{|y|\le1}\big(e^{i\langle u,y\rangle}-1\big)\big(\frac{1}{|y|^{\beta+b}}-1\big)\d y\Big)\Big](x),\\
  h_{f}(x):=&\F^{-1}\Big[\exp\Big(\frac{1}{\xi_b}\int_{\R^{b}}\big(e^{i\langle u,y\rangle}-1\big)\frac{1}{|y|^{\beta+b}\vee1}\d y\Big)\Big](x),\quad x\in\R^{b},
  \end{align*}
  are the densities of an infinitely divisible distribution with Lévy density $\nu_{c}(x):=\frac{1}{\xi_b}(\frac{1}{|x|^{\beta+b}}-1)\1_{\{|x|\le1\}}$
  and an infinitely divisible distribution with Lévy density $\nu_{f}(x):=\frac{1}{\xi_b(|x|^{\beta+b}\vee1)},\  x\in\R^{b}$,
  respectively. Since $\nu_{f}$ is integrable, $h_{f}$ is the
  density of a compound Poisson distribution which can be written as
  convolution exponential{, cf. \cite[Remark 27.3]{sato1999levy}},
  \begin{equation}\label{eq:convExp}
  h_{f}=e^{-\nu_{f}(\R^{b})}\sum_{j=0}^{\infty}\frac{\nu_{f}^{\ast j}}{j!}
  \end{equation}
  where $\nu_{f}^{\ast j}$ denotes the $j$-fold convolution of $\nu_{f}$, and {$\nu_{f}^{\ast 0}:=\delta_0$ is the Dirac measure in zero}. Therefore,  
  \begin{equation}
  f_0=e^{-\nu_{f}(\R^{b})}\sum_{j=0}^{\infty}\frac{h_{c}\ast\nu_{f}^{\ast j}}{j!}\ge e^{-\nu_{f}(\R^{b})}\big(h_{c}\ast\nu_{f}\big),\label{eq:lowBound0hyp}
  \end{equation}
 where, with some abuse of notation,  $\nu_{f}(\R^{b})$ stands for the total mass of $\nu_f$. Due to the compactness of the support of the Lévy measure corresponding to $h_{c}$, the density $h_{c}$ admits a finite
  exponential moment \citep[Theorem 26.1]{sato1999levy}, that is there exists $\alpha>0$ such that $\int_{\R^{b}}e^{\alpha|y|}h_{c}(y)\d y<\infty$.
  
   For any $x\neq0$, we have
  \begin{align}
  & |h_{c}\ast\nu_{f}(x)-\nu_{f}(x)| =\Big|\int_{\R^{b}}\big(\nu_{f}(x-y)-\nu_{f}(x)\big)h_{c}(y)\d y\Big|\nonumber \\
  & \le\int_{|y|\le|x|/2}\big|\nu_{f}(x-y)-\nu_{f}(x)\big|h_{c}(y)\d y+\int_{|y|>|x|/2}\big|\nu_{f}(x-y)-\nu_{f}(x)\big|h_{c}(y)\d y\nonumber
  \intertext{Rewriting $\nu_f(x)=\mu(|x|)$ with $\mu(r):=\xi_b^{-1}(r^{-\beta-b}\wedge1)$, we see that the expression in the last display does not exceed}
  & \sup_{|v|\le|x|}|\mu'(|v|)|\int_{\R^{b}}|y|h_{c}(y)\d y+2\|\nu_{f}\|_{\infty}e^{-\alpha|x|/2}\int_{|y|>|x|/2}e^{\alpha|y|}h_{c}(y)\d y\nonumber \\
  & \le\Big(\sup_{|v|\le|x|}|\mu'(|v|)|+2\|\nu_{f}\|_{\infty}e^{-\alpha|x|/2}\Big)\int_{\R^{b}}(|y|\vee e^{\alpha|y|})h_{c}(y)\d y.\label{eq:decayDenisty}
  \end{align}
  By the polynomial decay of $\mu$ we have $|\mu'(|x|)|=o(\nu_{f}(x))$ as $|x|\to\infty$ implying that $|h_{c}\ast\nu_{f}(x)-\nu_{f}(x)|=\nu_{f}(x)o(1)$ as $|x|\to\infty$.
  Combining (\ref{eq:lowBound0hyp}) and (\ref{eq:decayDenisty}) yields
  \begin{equation}
  \label{eq:tailsDenisty}
  f_0(x)\ge e^{-\nu_{f}(\R^{b})}\big(h_{c}\ast\nu_{f}\big)(x)=e^{-\nu_{f}(\R^{b})}\nu_{f}(x)\big(1+o(1)\big)\ge\frac{C}{|x|^{\beta+b}}, \quad \forall |x|>r,
  \end{equation}
 where $C>0$ and $r>0$ are constants depending only on $b$. 
  
  {From the representation \eqref{eq:convExp} we see that $f_0$ is strictly positive. By the decay of its characteristic function, $f_0$ is also continuous. Together with \eqref{eq:tailsDenisty}, we conclude that
  \[
    f(x)\ge C_b(1+|x|^{\beta+b})^{-1}
  \]
  for $\beta\in(0,1),R=0,T=1$.}
  
  In the case $\beta\in[1,2),R=0,T=1$ the proof is analogous with the only
  difference that the convolution exponential $h_{f}$ is shifted by
  $a:=(\int_{\R^{b}}x_{1}\nu_{f}(x)\d x,\dots,\int_{\R^{b}}x_{b}\nu_{f}(x)\d x)\in\R^{b}$,
  i.e.,
  \[
  h_{f}=e^{-\nu_{f}(\R^{b})}\delta_{a}\ast\Big(\sum_{j=0}^{\infty}\frac{\nu_{f}^{\ast j}}{j!}\Big)
  \]
  where
  $\delta_{a}$ is the Dirac distribution at~$a$. Thus, we replace everywhere above $g_{b}\ast h_{c}$ by $g_{b}\ast h_{c}\ast\delta_{a}$. Clearly, the argument remains valid with such a modification.
  
  {
  \emph{Step 2.} We now denote by $f$ the density $f_0$ from Step~1 corresponding to $R=0, T=1$. Thus, $f$ is a density with characteristic function $e^{-C |u|^\beta}$ for some $C>0$. With this notation, for $R=0, T>0$ we have $f_0(x)=\F^{-1}[e^{-C T|u|^\beta}](x)=T^{-b/\beta}f(T^{-1/\beta}x)$. We now turn to the case  $R>0, T>0$. 
  Denoting the density of the normal distribution $\mathcal N(0,I_b)$  by $g$ and using the lower bound from Step~1, we obtain  
   \begin{align*}
    f_0(x)&=\Big(\big(T^{-b/\beta}f(T^{-1/\beta}\cdot)\big)\ast\big(R^{-b/2}g(R^{-1/2}\cdot)\big)\Big)(x)\\
    &\ge (2\pi R)^{-b/2}\int_{\R^b}\frac{C_b}{1+|T^{-1/\beta}x-y|^{b+\beta}}e^{-\frac{T^{2/\beta}}{2R}|y|^2}\d y\\
    &\ge C_b(2\pi R)^{-b/2}\frac{1}{(1+2^{b+\beta-1}T^{-1-b/\beta}|x|^{b+\beta})}\int_{\R^b}\frac{1}{1+2^{\beta+b-1}|y|^{b+\beta}}e^{-\frac{T^{2/\beta}}{2R}|y|^2}\d y\\
    &\ge \frac{2\pi^{b/2}C_b}{4^{b+\beta}{\Gamma(b/2)}}(2\pi R)^{-b/2}\frac{1}{(1+T^{-1-b/\beta}|x|^{b+\beta})}\int_{\R_+}\frac{r^{b-1}}{1+r^{b+\beta}}e^{-\frac{T^{2/\beta}}{2R}r^2}\d r,
  \end{align*}
  where in the third line we have used the fact that $b+\beta>1$ and the convexity of $|x|^{b+\beta}$).
  Using the assumption $(\log n)^{-c}\le CR^{\beta/2}T^{-1}$, we deduce that with some constants $\bar C_b,C_b'>0$ depending only on $b$,
  \begin{align*}
    f_0(x)&\ge \bar C_b R^{-b/2}\frac{1}{(1+T^{-1-b/\beta}|x|^{b+\beta})}\int_{\R_+}\frac{r^{b-1}}{1+r^{b+\beta}}e^{-(C(\log n)^c)^{2/\beta}r^2/2}\d r\\
    &\ge \bar C_bR^{-b/2}(C(\log n)^c)^{-b/\beta}\frac{1}{(1+T^{-1-b/\beta}|x|^{b+\beta})}\int_{\R_+}\frac{r^{b-1}}{1+(C(\log n)^c)^{-1-b/\beta}r^{b+\beta}}e^{-r^2/2}\d r\\
    &\ge C_b'R^{-b/2}(\log n)^{-cb/\beta}\frac{1}{(1+T^{-1-b/\beta}|x|^{b+\beta})}\int_{\R_+}\frac{r^{b-1}e^{-r^2/2}}{1+r^{b+\beta}}\d r.
  \end{align*}
  Since the last integral is finite and positive we obtain the result of the lemma for $R>0, T>0$.
  }
\end{proof}
{
Due to Lemma~\ref{lem:lowBoundDensity} and the assumption $T(\log n)^{-1+\beta/2}\le C_3R^{\beta/2}$, the $\chi^{2}$-divergence \eqref{eq:Chi2} satisfies
\begin{align}\label{eq:boundChi2-1}
\chi^2(\P_{j,l},\P_{0})\le&C R^{b/2}(\log n)^{(1-\beta/2)b/\beta}\Big(\int_{\R^b}\big(f_{j,l}(x)-f_{0}(x)\big)^{2}\d x\\
&\qquad\qquad+\frac 1{T^{1+b/\beta}}\int_{\R^b}|x|^{\beta+b}\big(f_{j,l}(x)-f_{0}(x)\big)^{2}\d x\Big)\notag.
\end{align}}
We now bound separately the first and the second integral in \eqref{eq:boundChi2-1}.  Using Plancherel's identity, we rewrite the first integral as
\begin{align}\label{eq:planch}
\int_{\R^b}\big(f_{j,l}(x)-f_{0}(x)\big)^{2}\d x
 & =\frac{1}{(2\pi)^{b}}\big\|\phi_{j,l}-\phi_{0}\big\|_{L^{2}}^{2},
\end{align}
where $\phi_{j,l}$ and $\phi_{0}$ are the characteristic functions corresponding to $f_{j,l}$ and $f_0$, respectively.
We now consider the difference of the characteristic exponents 

\[
\eta_{j}(u):=\log\phi_{j,l}(u)-\log\phi_{0}(u)=-\frac12\langle u,\bar A_{j,l}u\rangle\big(1-\F K(\delta_{n,p}u)\big),
\]
where $\bar A_{j,l}= (\rho T\delta^{2-\beta}_{n,p}/b) A_{j,l}$.
A first order Taylor expansion yields 
\begin{align*}
\phi_{j,l}(u)-\phi_{0}(u) & =\eta_{j}(u)\phi_{0}(u)\int_{0}^{1}e^{t\eta_{j}(u)}\d t\\
&=\eta_j(u)e^{-R|u|^2/2}\int_0^1\psi^{1-t}_{0,l}(u)\psi^t_{j,l}(u)\exp\big(-\frac t2\langle u, \bar A_{j,l}u\rangle\big)\d t.
\end{align*}
Due to the property $1-\F K(\delta_{n,p}u)=0$ for $|u|\le\delta_{n,p}^{-1}$ and the elementary inequality $x^2e^{|x|}\le\exp(3|x|)$, $\forall x\in\R$, we obtain
\begin{align*}
  \|\phi_{j,l}-\phi_{0}\|_{L^{2}}^{2}
  \le & \int_0^1\Big\|\eta_{j}(u)\exp\Big(-\frac R2|u|^2-\frac t2\langle u,\bar A_{j,l} u\rangle\Big)\Big\|_{L^{2}}^{2}\d t\nonumber \\
  \le& \frac{1}4 \int_0^1\int_{|u|>1/\delta_{n,p}}|\langle u,\bar A_{j,l}u\rangle|^{2}\exp\Big(-R|u|^2-t\langle u,\bar A_{j,l}u\rangle\Big)\d u\,\d t\nonumber \\
  \le &  \frac{1}4 \int_{|u|>1/\delta_{n,p}}\exp\Big(-R|u|^2+3|\langle u,\bar A_{j,l}u\rangle|\Big)\d u\nonumber\\
  \le &  \frac{1}4 \int_{|u|>1/\delta_{n,p}}\exp\Big(-(R-3\rho T\delta_{n,p}^{{2-\beta}})|u|^2\Big)\d u,
\end{align*}
where we have used the bound $\|A_{j,l}\|_\infty\le b$. Finally, we choose $\rho>0$ sufficiently small to satisfy $3\rho T\delta_{n,p}^{{2-\beta}}\le R/4$. Then, 
\begin{align}
  \|\phi_{j,l}-\phi_{0}\|_{L^{2}}^{2}
  &\le \frac{1}4
  e^{- R/(4\delta_{n,p}^{2}) }\int_{|u|>\delta_{n,p}^{-1}}e^{-R|u|^2/2}\d u \le  \frac{1}4 \Big(\frac{2\pi}{R}\Big)^{b/2}
  e^{- R/(4\delta_{n,p}^{2}) }.\label{eq:l2norm}
\end{align}
{
To take into account the first factor in \eqref{eq:boundChi2-1}, we note that the definition of $\delta_{n,p}$ in \eqref{eq:deltaN} imply
\begin{equation}\label{eq:techn}
  R^{b/2}(\log n)^{(1-\beta/2)b/\beta}R^{-b/2}e^{- R/(12\delta_{n,p}^{2}) }\le (\log n)^{(1-\beta/2)b/\beta}\Big(\frac {\rho'\log(1+c_1' bp/k)}n\Big)^{1/2},
\end{equation}
where the last expression is uniformly bounded by a constant. Combining this remark with \eqref{eq:planch} and \eqref{eq:l2norm}, we finally get that there is a constant $C>0$ such that
\begin{align}
 R^{b/2}(\log n)^{(1-\beta/2)b/\beta}\int_{\R^b}\big(f_{j,l}(x)-f_{0}(x)\big)^{2}\d x
  &\le C \exp \Big(-\frac{R}{6\delta_{n,p}^2}\Big).\label{eq:l2normm}
\end{align}
}

To bound the second integral in \eqref{eq:boundChi2-1} we use the following proposition proved in the supplementary material in Appendix~\ref{supplement}.
\begin{prop}\label{prop:aux}
  There is a constant $C>0$ depending only on the kernel $K$ and on $b$ such that, for all $\beta\in(0,2)$ and $j=1,\dots,M$, $l=1,\dots,L$,
  \begin{equation}\label{eq:boundIntDeriv}
    \int_{\R^{b}}\xi_{b}|x|^{\beta+b}\big(f_{j,l}(x)-f_{0}(x)\big)^{2}\d x\le C{(1\vee R^{\beta/2})}\exp\Big(-\frac{R}{5\delta_{n,p}^2}\Big) . 
  \end{equation}
\end{prop}
{
This proposition and the assumption $T(\log n)^{c_*}\ge 1\vee R^{\beta/2}$ yield, via an argument analogous to \eqref{eq:techn}, that 
\begin{align}
&(\log n)^{(1-\beta/2)b/\beta}\frac {R^{b/2}}{T^{1+b/\beta}}\int_{\R^b}|x|^{\beta+b}\big(f_{j,l}(x)-f_{0}(x)\big)^{2}\d x\notag\\
  &\qquad\le C'(\log n)^{(1-\beta/2)b/\beta}\frac {R^{b/2}\vee R^{(b+\beta)/2}}{T^{(b+\beta)/\beta}} \exp\Big(-\frac{R}{5\delta_{n,p}^2}\Big)
  \le C\exp \Big(-\frac{R}{6\delta_{n,p}^2}\Big)\label{eq:boundIntDeriv1}
\end{align}
where $C,C'>0$ are constants.}
Combining \eqref{eq:boundChi2-1}, \eqref{eq:l2normm} and \eqref{eq:boundIntDeriv1} and using the definition of $\delta_{n,p}$ in \eqref{eq:deltaN}, we find
\begin{align}
  \chi^{2}(\P_{j,l},\P_{0}) 
  &\le  C\exp \Big(-\frac{R}{6\delta_{n,p}^2}\Big)
  \le  C \frac {\rho'\log(1+c_1' bp/k)}n
  \le \frac {C\rho'\log M}{C_1'kn}:= \frac {\rho''\log M}{kn},\label{eq:boundChi2-2}
  \end{align}
where the last inequality follows from \eqref{eq:numAlternatives}.
Taking into account \eqref{eq:chi2}  and \eqref{eq:boundChi2-2} we get
\begin{equation}\label{eq:boundChi2-3}
  \chi^{2}(\P^{\otimes n}_{\Sigma_{j},\psi_{j}},\P^{\otimes n}_{\Sigma_{0},\psi_{0}})
  \le\exp\Big(Ln\max_l\chi^{2}(\P_{j,l},\P_{0})\Big)-1\le \exp\Big(\frac{p\rho''\log M}{bk}\Big )-1\le M^{2\rho''/b}-1,
\end{equation}
where we have used that, by construction, $L=p/b$ and $p\le 2k$. Finally, choose $\rho'$ (and thus $\rho''$) small enough to guarantee that $M^{2\rho''/b}-1\le M/3$. 
Hence, condition (ii) is verified, which concludes the proof of the theorem. \qed

\appendix

\section{Proof of Theorem~\ref{thm:concentration}}
  For the later reference we set $\xi_{U}: =\inf_{|u|\le U/2}|\phi(u)|$ and introduce the events
  \[
    \Omega(u):=\big\{|\phi_n(u)-\phi(u)|\le |\phi(u)|/2\big\},\qquad u\in\R^p.
  \]
  Due to the decomposition \eqref{eq:regression}, we obtain the bound
  \begin{eqnarray}\label{eq:decomp}
  \|\mathcal{R}_{n}\|_{\infty}&\le & \,S^{(1)}_{n}+S^{(2)}_n+D_{n}, \\ \text{where} \ \
  S^{(1)}_{n}&:=&\int_{\R^{p}}\big|\Re\big(\log\phi_{n}(u)\1_{\{|\phi_n(u)|\ge\iota\}}-\log\phi(u)\big)\big|\1_{\Omega(u)}\frac{\|\Theta(u)\|_\infty}{|u|^2}w_{U}(u)\d u,\notag\nonumber \\
   S^{(2)}_n&:= & \int_{\R^{p}}\big|\log|\phi_{n}(u)|\1_{\{|\phi_n(u)|\ge\iota\}}-\log|\phi(u)|\big|\1_{\Omega(u)^c}\frac{\|\Theta(u)\|_\infty}{|u|^2}w_{U}(u)\d u,\notag\nonumber\\
    D_{n}&:=&\int_{\R^{p}}\big|\log|\psi(u)|\big|\frac{\|\Theta(u)\|_{\infty}}{|u|^{2}}w_{U}(u)\d u.\nonumber
  \end{eqnarray}
  Here, $S^{(i)}_n$ are stochastic error terms and $D_n$ is a deterministic error term.
  Using the decay of $\psi\in\mathcal H'_\beta(T)$, the form of the support of $w_U$ and the fact that $\|\Theta(u)\|_{\infty}=1$, we obtain
   \begin{equation}\label{r0}
    D_n\le 16U^{-2}\sup_{|u| \le U/2}\big|\log|\psi(u)|\big|\int_{\R^p}\frac{w(v)}{|v|^2}\d v\le C(w)TU^{-(2-\beta)}
  \end{equation}
  for a constant $C(w)>0$ depending only on $w$. 
      
  To bound $S^{(1)}_n$ in \eqref{eq:decomp}, we first note that we have on $\Omega(u)$ under the assumption $\xi_U\ge 1/\sqrt{n}$
  \[
    |\phi_n(u)|\ge|\phi(u)|-|\phi_n(n)-\phi(u)|\ge |\phi(u)|/2\ge \frac1{2\sqrt{n}} =\iota,
  \]
 for all $u$ in the support of $w_U$. Thus, the indicator function $\1_{\{|\phi_n(u)|\ge\iota\}}$ in $S^{(1)}_{n}$ can be omitted. Linearizing the logarithm yields
  \begin{equation*}
  \log\phi_{n}(u)-\log\phi(u)=\log\Big(\frac{\phi_{n}(u)-\phi(u)}{\phi(u)}+1\Big)=\frac{\phi_{n}(u)-\phi(u)}{\phi(u)}+r_{n}(u)
  \end{equation*}
 with a residual $r_{n}$ satisfying on $\Omega(u)$
  \begin{equation}
    |r_{n}(u)|\le {\bar c}\Big|\frac{\phi_{n}(u)-\phi(u)}{\phi(u)}\Big|^{2}\label{eq:estRem}
  \end{equation}
  where $ {\bar c}>0$ is a constant. Hence, we have
  \begin{eqnarray}\label{eq:stochError}
    S^{(1)}_n&\le& L_n+T_n
    \end{eqnarray}
   where 
   \begin{eqnarray}L_n:=\int_{\R^{p}}\frac{|\phi_{n}(u)-\phi(u)|}{|u|^{2}|\phi(u)|}w_{U}(u)\d u, \quad
   T_n:=\quad \int_{\R^p}\frac{|r_n(u)|}{|u|^{2}}\1_{\Omega(u)}w_U(u)\d u.\nonumber
  \end{eqnarray}
  Here, $L_n$ is the linearized stochastic error and $T_n$ is a remainder. By the Cauchy-Schwarz inequality
 \begin{align}\label{r1}
 L_n&\le \frac{16}{U^2\xi_U}\int_{\R^p}|\phi_n(u)-\phi(u)| w_U(u)\d u
 \le \frac{16\varkappa_w^{1/2}}{U^2\xi_U}Z
 \end{align}
 with $\bar\varkappa_w=\|w\|_{L^1}$ and
 $$
 Z=Z(Y_1,\dots, Y_n)= \Big(\int_{\R^{p}}|\phi_{n}(u)-\phi(u)|^2w_{U}(u)\d u\Big)^{1/2}.
 $$
 Similarly, we deduce from \eqref{eq:estRem}
 \begin{equation}\label{r2}
 T_n\le\frac{16\bar c}{U^2}\int_{\R^p}\frac{|\phi_n(u)-\phi(u)|^2}{|\phi(u)|^2} w_U(u)\d u\le  \frac{16 \bar c}{U^2\xi_{U}^2} Z^2.
 \end{equation}
 Note that $Z$ satisfies the bounded difference condition
 $$
 \forall Y_i, Y_{i}' \in \R^{p}: \quad   |Z(Y_1,\dots, Y_{i-1}, Y_{i}', Y_{i+1}, \dots, Y_n)-Z(Y_1,\dots, Y_n)|\le 2\bar \varkappa_w^{1/2}/n
 $$
 By the bounded difference inequality \cite[Theorem 3.3.14]{gine_book} we get $\P(Z\ge \E(Z) +t)\le \exp(-\frac{nt^2}{4\bar \varkappa_w})$, for all $t>0$. Since  
 $\E (Z) \le (\bar \varkappa_w/n)^{1/2}$ this implies
 $$
 \P\Big(Z\ge \frac{\bar\varkappa_w^{1/2}}{\sqrt n}( 2\gamma+1)\Big)\le e^{- \gamma^2}, \quad \forall \gamma>0.
 $$ 
 Using \eqref{r1},\eqref{r2}, and the assumption that $\gamma\ge 1$ we find that there exists a numerical constant $c_1^*>0$ such that
  \begin{equation*}
 \P\Big( S^{(1)}_n\ge   \frac{c_1^*\bar\varkappa_w\gamma}{U^2\xi_{U}\sqrt n}\Big(1+\frac{\gamma}{\xi_U\sqrt n}\Big)\Big)\le 2e^{- \gamma^2}. \end{equation*}
 Since $\xi_U\le 1$, this implies
 \begin{equation}\label{r3}
 \P\Big( S^{(1)}_n\ge   \frac{2c_1^*\bar\varkappa_w\gamma^2}{U^2\xi_{U}^2\sqrt n}\Big)\le 2e^{- \gamma^2}. \end{equation}
 Using lower bounds $\iota$ and $\xi_U$ for $|\phi_n(u)|$ and $|\phi(u)|$, respectively, and applying the elementary bound $\1_{\{a>1\}}< a$ for any $a>0$,  the term $S^{(2)}_n$ in \eqref{eq:decomp} is bounded as follows:
  \begin{align*}
    S^{(2)}_{n}\le& \int_{\R^{p}}\big(\log \iota^{-1}+\log\xi_U^{-1}\big)\1_{\Omega(u)^c}\frac{w_U(u)}{|u|^2}\d u\\
    \le& 2\int_{\R^{p}}\big(\log \iota^{-1}+\log\xi_U^{-1}\big)\frac{|\phi_n(u)-\phi(u)|}{|\phi(u)|}\frac{w_U(u)}{|u|^2}\d u\\
    \le& \frac{32(2\log\xi_U^{-1}+\log 2)}{U^2\xi_U}\int_{\R^{p}}|\phi_n(u)-\phi(u)|w_U(u)\d u
    \le \frac{32\sqrt{\bar\varkappa_w}}{U^2\xi_U^2}Z.
  \end{align*}
  Hence, for some numerical constant $c_2^*>0$ we have
  \begin{align}\label{r4}
    \P\Big(S^{(2)}_n\ge \frac{c_2^*\bar\varkappa_w}{U^2\xi_U^2\sqrt n}\gamma\Big)\le e^{- \gamma^2}, \quad \forall \gamma>0.
  \end{align}
 Combining \eqref{eq:decomp},  \eqref{r0}, \eqref{r3} and \eqref{r4} and using the fact that $\gamma\ge 1$ we obtain 
   \begin{align*}
  & \mathbb{P}\Big(\|\mathcal{R}_{n}\|_{\infty}\ge \frac{(2c_1^*+c_2^*)\bar\varkappa_w\gamma^2}{U^2\xi_{U}^2\sqrt n}+C(w)TU^{-2+\beta}\Big)\le3e^{-\gamma^{2}}.
  \end{align*}
  Finally, we use the bound $\xi_U\ge \exp(-\|\Sigma\|_\infty U^2/8-2TU^{\beta})$ that is shown similarly to the analogous bound in the proof of Theorem~\ref{thm:ConcentrationSigma*}. 

\section{Proofs for Section~\ref{sec:elliptic}}
\begin{proof}[Proof of Lemma~\ref{lem:taylorEta}]
  By Taylor's formula we have for some $\xi\in[0,1]$ that
  \begin{align*}
     R(u):=&\eta^{-1}\big(-\log|\phi|(u)\big)-\langle u,\Sigma u\rangle-\frac{\log|\psi|(u)}{\eta'(\langle u,\Sigma u\rangle)}\\
     =&\frac{(\log|\psi|(u))^2}{2}(\eta^{-1})''\big(\eta\big(\langle u,\Sigma u\rangle\big)-\xi\log|\psi|(u)\big)
  \end{align*}
  Since $(\eta^{-1})''(x)=-\eta''(\eta^{-1}(x))/\eta'(\eta^{-1}(x))^3$ and thus $|(\eta^{-1})''(x)|\le T|\eta^{-1}(x)|^{-1}|\eta'(\eta^{-1}(x))|^{-2}$ we have
  \begin{align}\label{eq:remElliptic}
    |R(u)|\le\frac{T|\log|\psi|(u)|^2}{2|g(u,\xi)||\eta'(\eta^{-1}(\eta(\langle u,\Sigma u\rangle\big)-\xi\log|\psi|(u)))|^2}
  \end{align}
  with $g(u,\xi):=\eta^{-1}(\eta(\langle u,\Sigma u\rangle\big)-\xi\log|\psi|(u))$. Since $\eta^{-1}$ is non-negative and monotone increasing and $\log |\psi(u)|\le 0$, we have
  \[
    |g(u,\xi)|=g(u,\xi)\ge\eta^{-1}\big(\eta(\langle u,\Sigma u\rangle)\big)=\langle u,\Sigma u\rangle.
  \]
  Another Taylor expansion for the second term in the denominator in \eqref{eq:remElliptic} yields for some $\xi'\in[0,1]$
  \begin{align}
    &\eta'(\eta^{-1}(\eta\big(\langle u,\Sigma u\rangle\big)-\xi\log|\psi|(u))\notag\\
    &\qquad=\eta'(\langle u,\Sigma u\rangle)-\xi(\log|\psi(u)|)\,(\eta'\circ\eta^{-1})'(\eta(\langle u,\Sigma u\rangle\big)-\xi'\log|\psi|(u))\notag\\
    &\qquad=\eta'(\langle u,\Sigma u\rangle)+\xi(\log|\psi(u)|)\,\Big(\frac{\eta''\circ\eta^{-1}}{\eta'\circ\eta^{-1}}\Big)(\eta(\langle u,\Sigma u\rangle\big)-\xi'\log|\psi|(u))\notag\\
    &\qquad\ge\eta'(\langle u,\Sigma u\rangle)-T|\log|\psi|(u)|/g(u,\xi')\label{eq:ellipticBound}\\
    &\qquad\ge\eta'(\langle u,\Sigma u\rangle)\big(1-T^2(1+|u|)^{\beta}/\langle u,\Sigma u\rangle\big)\notag.
  \end{align}
  If $|u|\ge(2^{\beta+1}T^2/\lambda_{min})^{1/(2-\beta)}$, then we conclude
  \[
    |R(u)|\le\frac{2|\log|\psi|(u)|^2}{\langle u,\Sigma u\rangle\eta'(\langle u,\Sigma u\rangle)^2}\le \frac{4T^2}{\lambda_{min}}|u|^{2\beta-2}.\tag*{\qedhere}
  \]
\end{proof}

\begin{proof}[Proof of Proposition~\ref{prop:ConcentrationSigmaPhi}]
  Due to Lemma~\ref{lem:taylorEta} and the mean value theorem, the estimation error can be bounded for any $U\ge(2^{\beta+1}T^2/\lambda_{min})^{1/(2-\beta)}$ by
  \begin{align*}
     |\hat\sigma^\Phi_{i,i}-\sigma_{i,i}|&\le U^{-2}\Big|\eta^{-1}\big(-\log|\phi_{n}(Uu^{(i)})|\big)-\eta^{-1}\big(-\log|\phi(Uu^{(i)})|\big)\Big|+\big(2T+\frac{4T^2}{\lambda_{min}}U^{-2+\beta}\big)U^{-2+\beta}\\
     &=\frac{|S(Uu^{(i)})|}{U^2\eta'(\eta^{-1}(-\log|\phi(Uu^{(i)})|+\xi S(Uu^{(i)})))} +\big(2T+2\big)U^{-2+\beta}
  \end{align*}
  for some $\xi\in[0,1]$ and $S(u)=\log|\phi_n(u)|-\log|\phi(u)|$ from Lemma~\ref{lem:reduction}. As in \eqref{eq:ellipticBound} (for any $u$ with $g(u,\xi)\ge\langle u,\Sigma u\rangle\ge1$), we deduce
  \begin{align*}
    \eta'\big(\eta^{-1}(-\log|\phi(u)|+\xi S(u))\big)\ge\eta'(\langle u,\Sigma u\rangle)-T(|\log|\psi|(u)|+S(u)|.
  \end{align*}
  On the event $\{|\log|\psi|(u)|+S(u)|\le \eta'(U^2\sigma_{ii})/2\}$, we thus obtain
  \[
    |\hat\sigma^\Phi_{i,i}-\sigma_{i,i}|\le\frac{2|S(Uu^{(i)})|}{U^2\eta'(U^2\sigma_{ii})} +\big(2T+2\big)U^{-2+\beta}.
  \]
  From this line, the argument is analogous to the proof of Theorem~\ref{thm:ConcentrationSigma*}.\qedhere
\end{proof}

\section{Supplementary material. Proof of Proposition~\ref{prop:aux}}\label{supplement}
Our aim is to prove the bound \eqref{eq:boundIntDeriv}. We fix $l$, and we will omit for brevity the index $l$ throughout the proof. This will cause a slight change of notation as compared to the proof of Theorem~\ref{thm:lowerBound} since, in what follows, $\psi_j:=\psi_{j,l}$, while $\psi_j$ was a product of $\psi_{j,l}$'s in the proof of Theorem~\ref{thm:lowerBound}. In addition, we will use the notation $S_j = RI_b+\bar A_j$, $S_0=RI_b$, $j=1,\dots, M$. Throughout $C>0$ will denote a generic constant whose value may change from line to line.

By construction, the characteristic functions $\phi_{j}(u)$ and $\phi_{0}(u)$ coincide on $\{|u|\le1/\delta_{n,p}\}$. 
For sufficiently small $\rho$ we have $\big|\langle u, \bar A_{j} u\rangle\big|\le\rho T\delta_{n,p}^{2-\beta}|u|^{2}\le R|u|^{2}/2$
implying $\langle u, S_{j} u\rangle=R|u|^{2}+\langle u, \bar A_{j} u\rangle \ge R|u|^{2}/2$.
Therefore, we can write
\begin{align*}
&\phi_{j}(u)-\phi_{0}(u)\\
&\quad=\big(\phi_{S_{j}}(u)\psi_{j}(u)\1_{\{|S_{j}^{1/2}u|>(R/2)^{1/2}\delta_{n,p}^{-1}\}}-\phi_{S_{0}}(u)\psi_0(u)\1_{\{|S_{0}^{1/2}u|>(R/2)^{1/2}\delta_{n,p}^{-1}\}}\big)\1_{\{|u|>1/\delta_{n,p}\}}
\end{align*}
with $\phi_{S_{j}}$ denoting the characteristic function of $\mathcal{N}(0,S_{j})$. We have 
\[
\xi_{b}\int_{\R^{b}}|x|^{\beta+b}\big(f_{j}(x)-f_{0}(x)\big)^{2}\d x\le2\xi_{b}(T_{j}+T_{0})
\]
where (also for $j=0$)
\begin{align*}
T_{j} & =\int_{\R^{b}}|x|^{\beta+b}\big(\F^{-1}\big[\phi_{S_{j}}\1_{\{|S_{j}^{1/2}\cdot|>(R/2)^{1/2}\delta_{n,p}^{-1}\}}\big]\ast\F^{-1}\big[\psi_{j}\1_{\{|\cdot|>1/\delta_{n,p}\}}\big]\big)^{2}(x)\d x\nonumber \\
 & =\Big\||x|^{(\beta+b)/2}\big(\F^{-1}\big[\phi_{S_j}\1_{\{|S_j^{1/2}\cdot|>(R/2)^{1/2}\delta_{n,p}^{-1}\}}\big]\ast\F^{-1}\big[\psi_{j}\1_{\{|\cdot|>1/\delta_{n,p}\}}\big]\big)\Big\|_{L^{2}}^{2}\nonumber \\
 & \le2^{\beta+b-1}\Big(\Big\||x|^{(\beta+b)/2}\F^{-1}\big[\phi_{S_j}\1_{\{|S_j^{1/2}\cdot|>(R/2)^{1/2}\delta_{n,p}^{-1}\}}\big]\Big\|_{L^{2}}^{2}\Big\|\F^{-1}\big[\psi_{j}\1_{\{|\cdot|>1/\delta_{n,p}\}}\big]\Big\|_{L^{1}}^{2}\nonumber \\
 & \qquad+\Big\|\F^{-1}\big[\phi_{S_j}\1_{\{|S_j^{1/2}\cdot|>(R/2)^{1/2}\delta_{n,p}^{-1}\}}\big]\Big\|_{L^{1}}^{2}\Big\||x|^{(\beta+b)/2}\F^{-1}\big[\psi_{j}\1_{\{|\cdot|>1/\delta_{n,p}\}}\big]\Big\|_{L^{2}}^{2}\Big),
\end{align*}
using Young's inequality in the last estimate.
Let us simplify this upper bound. 
We have the identity 
\[
\F^{-1}\big[\phi_{S_j}(u)\1_{\{|S_j^{1/2}u|>(R/2)^{1/2}\delta_{n,p}^{-1}\}}\big](x)=\frac{1}{\sqrt{\det S_j}}\F^{-1}\big[\phi_{I_{b}}\1_{\{|u|\ge (R/2)^{1/2}\delta_{n,p}^{-1}\}}\big](S_j^{-1/2}x)
\]
such that
\begin{align*}
&\Big\||x|^{(\beta+b)/2}\F^{-1}\big[\phi_{S_j}\1_{\{|S_j^{1/2}\cdot|>(R/2)^{1/2}\delta_{n,p}^{-1}\}}\big]\Big\|_{L^{2}}^{2} \\
&\qquad =\frac1{\sqrt{\det S_j}}\Big\||S_j^{1/2}x|^{(\beta+b)/2}\F^{-1}\big[\phi_{I_{b}}\1_{\{|\cdot|>(R/2)^{1/2}\delta_{n,p}^{-1}\}}\big]\Big\|_{L^{2}}^{2}\\
 &\qquad \le\frac{\|S_j\|_{\infty}^{(\beta+b)/2}}{\sqrt{\det S_j}}\Big\||x|^{(\beta+b)/2}\F^{-1}\big[\phi_{I_{b}}\1_{\{|\cdot|>(R/2)^{1/2}\delta_{n,p}^{-1}\}}\big]\Big\|_{L^{2}}^{2}
\end{align*}
and
\[
\|\F^{-1}[\phi_{S_j}\1_{\{|S_j^{1/2}\cdot|>(R/2)^{1/2}\delta_{n,p}^{-1}\}}]\|_{L^{1}}^{2}=\|\F^{-1}[\phi_{I_{b}}\1_{\{|\cdot|>(R/2)^{1/2}\delta_{n,p}^{-1}\}}]\|_{L^{1}}^{2}.
\]
Since $\rho\delta_{n,p}$ is small, the construction of $S_j$ as pertubation of $RI_{b}$ implies that $\|S_j\|_{\infty}$ is of the order $R$ 
and $\det S_j$ is of the order $R^{b}$. Hence, $(\det S_j)^{-1/2}\|S_j\|_{\infty}^{(\beta+b)/2}\le CR^{\beta/2}$. We deduce
\begin{align*}
T_{j} & \le C 2^{\beta+b}\Big(R^{\beta/2}\Big\||x|^{(\beta+b)/2}\F^{-1}\big[\phi_{I_b}\1_{\{|\cdot|>(R/2)^{1/2}\delta_{n,p}^{-1}\}}\big]\Big\|_{L^{2}}^{2}
   \Big\|\F^{-1}\big[\psi_j\1_{\{|\cdot|>\delta_{n,p}^{-1}\}}\big]\Big\|_{L^{1}}^{2}\nonumber \\
 & \qquad\qquad\quad+\Big\|\F^{-1}\big[\phi_{I_b}\1_{\{|\cdot|>(R/2)^{1/2}\delta_{n,p}^{-1}\}}\big]\Big\|_{L^{1}}^{2}\Big\||x|^{(\beta+b)/2}\F^{-1}\big[\psi_{j}\1_{\{|\cdot|>1/\delta_{n,p}\}}\big]\Big\|_{L^{2}}^{2}\Big).
\end{align*}
By the Cauchy-Schwarz inequality we further estimate
\begin{align*}
&\|\F^{-1}[\phi_{I_{b}}\1_{\{|\cdot|>(R/2)^{1/2}\delta_{n,p}^{-1}\}}]\|_{L^{1}}^{2} \\
& \quad\le\big\|(1+|x|^{(b+\beta)/2})^{-1}\big\|_{L^{2}}^{2}\big\|(1+|x|^{(b+\beta)/2})\F^{-1}[\phi_{I_{b}}\1_{\{|\cdot|>(R/2)^{1/2}\delta_{n,p}^{-1}\}}]\big\|_{L^{2}}^{2}\\
 & \quad\le2\big\|(1+|x|^{(b+\beta)/2})^{-1}\big\|_{L^{2}}^{2}\\
 &\qquad\times\Big(\|\F^{-1}[\phi_{I_{b}}\1_{\{|\cdot|>(R/2)^{1/2}\delta_{n,p}^{-1}\}}]\|_{L^{2}}^{2}+\||x|^{(b+\beta)/2}\F^{-1}[\phi_{I_{b}}\1_{\{|\cdot|>(R/2)^{1/2}\delta_{n,p}^{-1}\}}]\|_{L^{2}}^{2}\Big).
\end{align*}
An analogous estimate can be applied to $\|\F^{-1}\big[\psi_j\1_{\{|\cdot|>\delta_{n,p}^{-1}\}}\big]\|_{L^{1}}^{2}$. By a polar coordinate transformation we have 
\[
\big\|(1+|x|^{(b+\beta)/2})^{-1}\big\|_{L^{2}}^{2}=\frac{2\pi^{b/2}}{\Gamma(b/2)}\int_{0}^{\infty}\frac{r^{b-1}}{(1+r^{(b+\beta)/2})^{2}}\d r\le\frac{2\pi^{b/2}}{\Gamma(b/2)}\int_{0}^{\infty}\big(\1_{\{r<1\}}+r^{(1+\beta)/2}\big)^{-2}\d r
\]
Recalling the definition of $\xi_{b}=c\frac{2\pi^{b/2}}{\Gamma(b/2)}$, we conclude with Plancherel's identity and the Laplace operator $\Delta$
\begin{align}
 & \xi_{b}\int_{\R^{b}}|x|^{\beta+b}\big(f_{j}(x)-f_{0}(x)\big)^{2}\d x\nonumber \\
\le & C2^{\beta+b}\xi_{b}\Big(\big\|\F^{-1}[\phi_{I_{b}}\1_{\{|\cdot|>(R/2)^{1/2}\delta_{n,p}^{-1}\}}]\big\|_{L^{2}}^{2}+(1\vee R^{\beta/2})\big\||x|^{(\beta+b)/2}\F^{-1}[\phi_{I_{b}}\1_{\{|\cdot|>(R/2)^{1/2}\delta_{n,p}^{-1}\}}]\big\|_{L^{2}}^{2}\Big)\nonumber \\
 & \qquad\times\sum_{k\in\{0,j\}}\xi_{b}\Big(\big\|\F^{-1}\big[\psi_{k}\1_{\{|\cdot|>1/\delta_{n,p}\}}\big]\big\|_{L^{2}}^{2}+\big\||x|^{(\beta+b)/2}\F^{-1}\big[\psi_{k}\1_{\{|\cdot|>1/\delta_{n,p}\}}\big]\big\|_{L^{2}}^{2}\Big)\nonumber \\
\le & \frac{C\xi_{b}2^{\beta+b}}{(2\pi)^{b}}\Big(\int_{|u|>(R/2)^{1/2}\delta_{n,p}^{-1}}\phi_{I_{b}}(u)^{2}\d u
   +(1\vee R^{\beta/2})\int_{|u|>(R/2)^{1/2}\delta_{n,p}^{-1}}\big(\Delta^{\lceil(\beta+b)/4\rceil}\phi_{I_{b}}(u)\big)^{2}\d u\Big)\nonumber\\
 & \qquad\times\sum_{k\in\{0,j\}}\frac{\xi_{b}}{(2\pi)^{b}}\Big(\int_{\R^b}|\psi_{k}(u)|^{2}\d u+\int_{\R^b}\big|\Delta^{\lceil(\beta+b)/4\rceil}\psi_{k}(u)\big|^{2}\d u\Big).\label{eq:decompChi2} 
\end{align}

To bound these integrals, we apply the following proposition concerning tail integrals for stable distributions.
\begin{prop}\label{prop:tailbounds}
Let $\alpha,\beta\in(0,2]$ and $t\ge0$. Then there is a constant $C_{\alpha}>0$
depending only on $\alpha$ such that
\begin{align*}
\int_{|x|>t}\big|\Delta^{\lceil(b+\beta)/4\rceil}\exp(-|x|^{\alpha}/\alpha)\big|^{2}\d x & \le C_{\alpha}\Gamma(b/2)(13\pi+25\pi\alpha)^{b/2}b^{8}e^{-t^\alpha/\alpha}.
\end{align*}
\end{prop}
\begin{proof}
Set $q=\lceil(b+\beta)/4\rceil$. Verified by induction, we have for
any function $g\colon\R\to\R$ and for the radius $r\colon\R^{b}\to\R,x\mapsto|x|=\sqrt{x_{1}^{2}+\ldots+x_{b}^{2}}$
the following identity for the Laplace operator applied to the radial
function $g\circ r$
\[
\Delta^{q}g(r)=\sum_{k=0}^{q}\frac{1}{2^{k}k!}\left(\Delta^{k}r^{2q}\right)\cdot\left(\frac{1}{r}\frac{\partial}{\partial r}\right)^{2q-k}g(r).
\]
Since $\Delta(x_{1}^{2}+\ldots+x_{b}^{2})^{k}=\gamma_{b,k}(x_{1}^{2}+\ldots+x_{b}^{2})^{k-1}$
with $\gamma_{b,k}:=2k(b+2(k-1)),$ we get 
\begin{align*}
\Delta^{k}r^{2q} & =\left(\prod_{j=0}^{k-1}\gamma_{b,q-j}\right)r^{2(q-k)}
\end{align*}
and as a result 
\begin{align*}
\Delta^{q}g(r) & =\sum_{k=0}^{q}\frac{\left(\prod_{j=0}^{k-1}\gamma_{b,q-j}\right)r^{2(q-k)}}{2^{k}k!}\cdot\left(\frac{1}{r}\frac{\partial}{\partial r}\right)^{2q-k}g(r)\\
 & =\sum_{k=0}^{q}\binom qk\left(\prod_{j=0}^{k-1}(b+2(q-j-1))\right)r^{2(q-k)}\cdot\left(\frac{1}{r}\frac{\partial}{\partial r}\right)^{2q-k}g(r).
\end{align*}
Since $b\le4q$, we can estimate for all $k\le q$
\begin{align*}
\prod_{j=0}^{k-1}(b+2(q-j-1))=2^{k}\prod_{j=1}^{k}(b/2+q-j) & \le2^{k}\prod_{j=1}^{k}(3q-j)\\
 & =2^{k}\Big(\prod_{j=1}^{k}\big(1+\frac{q}{2q-j}\big)\Big)\Big(\prod_{j=1}^{k}(2q-j)\Big)\\
 & \le\frac{4^{k}(2q)!}{(2q-k)!}
\end{align*}
and thus
\begin{equation}
\Delta^{q}g(r)\le\sum_{k=0}^{q}\binom qk\frac{4^{k}(2q)!}{(2q-k)!}r^{2(q-k)}\cdot\Big|\left(\frac{1}{r}\frac{\partial}{\partial r}\right)^{2q-k}g(r)\Big|\label{eq:firstBound}
\end{equation}

Let $g(r)=\exp(-r^{\alpha}/\alpha)$ for $\alpha\in(0,2)$. For any
$m\in\mathbb{N}$ we define a polynomial $P_{m}$ via
\[
P_{m}(r)g(r):=\Big(\frac{1}{r}\frac{\partial}{\partial r}\Big)^{m}g(r).
\]
It is easy to see that 
\[
P_{k}(r)=\sum_{l=1}^{k}(-1)^{l}\lambda_{k,l}r^{l\alpha-2k},
\]
with coefficients
\begin{align*}
\lambda_{k,1} & =-\prod_{j=1}^{k-1}(\alpha-2j),\qquad\lambda_{k,k}=1\\
\lambda_{k,l} & =(l\alpha-2(k-1))\lambda_{k-1,l}+\lambda_{k-1,l-1},\quad\mbox{for }l=2,\dots,k-1.
\end{align*}
From this recursion formula, we obtain the upper bound
\begin{equation}
|\lambda_{k,l}|\le(1+A^{-1})^{k}A^{l}\prod_{j=1}^{k-l}(2j+l(2-\alpha)-2)\label{eq:boundLambda}
\end{equation}
for all $k\ge1,l=1,\dots,k$ and any $A>0$. Indeed, (\ref{eq:boundLambda})
is satisfied for $\lambda_{k,1},\lambda_{k,k}$ for all $k$ and we
check inductively, given (\ref{eq:boundLambda}) holds true for $\lambda_{k-1,l}$,
\begin{align*}
|\lambda_{k,l}| & \le|2(k-l)+l(2-\alpha)-2|\,|\lambda_{k-1,l}|+|\lambda_{k-1,l-1}|\\
 & \le(1+A^{-1})^{k-1}\prod_{j=1}^{k-l}(2j+l(2-\alpha)-2)\big(A^{l}+A^{l-1}\big)\\
 & \le(1+A^{-1})^{k}A^{l}\prod_{j=1}^{k-l}(2j+l(2-\alpha)-2).
\end{align*}

From (\ref{eq:boundLambda}), we obtain the upper bound
\begin{align*}
\Big|\Big(\frac{1}{r}\frac{\partial}{\partial r}\Big)^{m}g(r)\Big| & \le e^{-r^{\alpha}/\alpha}(1+A^{-1})^{m}\sum_{l=1}^{m}A^{l}\prod_{j=1}^{m-l}(2j+l(2-\alpha)-2)r^{l\alpha-2m}\\
 & \le e^{-r^{\alpha}/\alpha}(2+2A^{-1})^{m}\sum_{l=1}^{m}\big(\frac{A}{2}\big)^{l}\prod_{j=l+1}^{m}(j-l\alpha/2)r^{l\alpha-2m}\\
 & \le e^{-r^{\alpha}/\alpha}(2+2A^{-1})^{m}\sum_{l=1}^{m}\frac{m!}{l!}\big(\frac{A}{2}\big)^{l}r^{l\alpha-2m}.
\end{align*}
Therefore, we get from (\ref{eq:firstBound}) 
\begin{align*}
\left|\Delta^{q}g(r)\right| & \leq\sum_{k=0}^{q}\binom qk\frac{(2q)!}{(2q-k)!}4^{k}r^{2(q-k)}e^{-r^{\alpha}/\alpha}(2+2A^{-1})^{2q-k}\sum_{l=1}^{2q-k}\frac{(2q-k)!}{l!}\big(\frac{A}{2}\big)^{l}r^{l\alpha-2(2q-k)}\\
 & =e^{-r^{\alpha}/\alpha}(2q)!\sum_{k=0}^{q}\sum_{l=1}^{2q-k}\underbrace{\binom qk 4^{k}(2+2A^{-1})^{2q-k}\big(\frac{A}{2}\big)^{l}}_{=:\gamma(k,l)}\frac{r^{l\alpha-2q}}{l!}.
\end{align*}
Using $b\le4q$, we find
\begin{align*}
\int_{|x|>t}|\Delta^{q}g(|x|)|^{2}dr & =\frac{2\pi^{b/2}}{\Gamma(b/2)}\int_{t}^{\infty}r^{b-1}\left|\Delta^{q}g(r)\right|^{2}\,\d r\\
 & \le\frac{2\pi^{b/2}(2q)!^{2}}{\Gamma(b/2)}\sum_{k=0}^{q}\sum_{l=1}^{2q-k}\sum_{k'=0}^{q}\sum_{l'=1}^{2q-k'}\frac{\gamma(k,l)\gamma(k',l')}{l!\,l'!}\int_{t}^{\infty}r^{p-1-4q+(l+l')\alpha}e^{-2r^{\alpha}/\alpha}\d r\\
 & \le\frac{2\pi^{b/2}(2q)!^{2}}{\Gamma(b/2)}\sum_{k=0}^{q}\sum_{l=1}^{2q-k}\sum_{k'=0}^{q}\sum_{l'=1}^{2q-k'}\frac{\gamma(k,l)\gamma(k',l')}{l!\,l'!}\int_{t}^{\infty}r^{(l+l')\alpha-1}e^{-2r^{\alpha}/\alpha}\d r.
\end{align*}
To bound the tail integrals, we apply the following: Substituting $s=r^{\alpha}/\alpha$, we obtain for any $m>0$ 
\begin{align}
\int_{t}^{\infty}r^{m\alpha-1}e^{-2r^{\alpha}/\alpha}\,\d r 
& \le e^{-t^\alpha/\alpha}\int_{0}^{\infty}r^{m\alpha-1}e^{-r^{\alpha}/\alpha}\,\d r\nonumber \\
 & =e^{-t^\alpha/\alpha}\alpha^{m-1}\int_{0}^{\infty}s^{m-1}e^{-s}\,\d s
  =e^{-t^\alpha/\alpha}\alpha^{m-1}\Gamma(m).\label{eq:tailIntegrals}
\end{align}
Together with $\frac{\Gamma(l+l')}{l!\,l'!}\le\binom{l+l'}{l}\le2^{l+l'}$,
we deduce from (\ref{eq:tailIntegrals})
\begin{align*}
\int_{|x|>t}|\Delta^{q}g(|x|)|^{2}dr & \le e^{-t^\alpha/\alpha}\frac{2\pi^{b/2}(2q)!^{2}}{\Gamma(b/2)}\sum_{k=0}^{q}\sum_{l=1}^{2q-k}\sum_{k'=0}^{q}\sum_{l'=1}^{2q-k'}\gamma(k,l)\gamma(k',l')\alpha^{l+l'}\frac{\Gamma(l+l')}{l!\,l'!}.\\
 & \le e^{-t^\alpha/\alpha}\frac{2\pi^{b/2}(2q)!^{2}}{\Gamma(b/2)}\sum_{k=0}^{q}\sum_{l=1}^{2q-k}\sum_{k'=0}^{q}\sum_{l'=1}^{2q-k'}\gamma(k,l)\gamma(k',l')(2\alpha)^{l+l'}\\
 & =e^{-t^\alpha/\alpha}\frac{2\pi^{b/2}(2q)!^{2}}{\Gamma(b/2)}\Big(\sum_{k=0}^{q}\sum_{l=1}^{2q-k}\gamma(k,l)(2\alpha)^{l}\Big)^{2}.
\end{align*}
The sum in the last bound is given by
\[
\sum_{k=0}^{q}\sum_{l=1}^{2q-k}\gamma(k,l)(2\alpha)^{l}=\sum_{k=0}^{q}\binom qk4^{k}(2+2A^{-1})^{2q-k}\sum_{l=1}^{2q-k}\big(A\alpha\big)^{l}
\]
where for all $A>\alpha^{-1}$ 
\[
\sum_{l=1}^{2q-k}\big(A\alpha\big)^{l}=\frac{(A\alpha)^{2q-k}-1}{1-1/(A\alpha)}\le\frac{(A\alpha)^{2q-k}}{1-1/(A\alpha)}.
\]
Hence,
\begin{align*}
\int_{|x|>t}|\Delta^{q}g(|x|)|^{2}dr & \le\frac{2e^{-t^\alpha/\alpha}}{\alpha-1/A}\frac{\pi^{b/2}(2q)!^{2}}{\Gamma(b/2)}\Big(\sum_{k=0}^{q}\binom qk4^{k}(2+2A^{-1})^{2q-k}\big(A\alpha\big)^{2q-k}\Big)^{2}\\
 & =\frac{2e^{-t^\alpha/\alpha}}{\alpha-1/A}\frac{\pi^{b/2}(2q)!^{2}}{\Gamma(b/2)}(2A\alpha+2\alpha)^{2q}\big(4+2(A+1)\alpha\big)^{2q}
\end{align*}
Finally, we apply $(2q)!=(b/2+3)!\le(b/2+3)^{4}\Gamma(b/2)$, due
to $q\le b/4+3/2$, to conclude
\[
\int_{|x|>t}|\Delta^{q}g(|x|)|^{2}dr\le C_{\alpha,A}\Gamma(b/2)\pi^{b/2}(2A\alpha+2\alpha)^{b/2}\big(4+2A\alpha+2\alpha\big)^{b/2}b^{8}e^{-t^\alpha/\alpha}
\]
for some constant $C_{\alpha,A}$ depending only on $\alpha$ and
$A$. Since $\alpha\le2$ there there is some $A>\alpha^{-1}$ such
that $(A\alpha+\alpha)(4+A\alpha+\alpha)\le13+25\alpha$ we obtain
the asserted upper bound 
\end{proof}

Applying Proposition~\ref{prop:tailbounds} and $\int_{|u|\ge t}e^{-|u|^2}\d u\le \xi_b2^{b/2-1}\Gamma(b/2) e^{-t^2/2}$ to the Gaussian components in \eqref{eq:decompChi2} and using $\frac{\xi_{b}}{(2\pi)^{b}}=\frac{c}{2^{b}\pi{}^{b/2}\Gamma(b/2)}$ from \eqref{eq:xi}, we obtain
\begin{align*}
 & \xi_{b}\int_{\R^{b}}|x|^{\beta+b}\big(f_{j}(x)-f_{0}(x)\big)^{2}\d x \\
 &\quad\le \frac{C\xi_{b}2^{b}}{(2\pi)^{b}}\Big((1\vee R^{\beta/2})\Gamma(b/2)(63\pi)^{b/2}b^{8}e^{-R\delta_{n,p}^{-2}/4}\Big)\\
 & \quad\quad\times\sum_{k\in\{0,j\}}\frac{C\xi_{b}}{(2\pi)^{b}}\Big(\int_{\R^b}|\psi_{k}(u)|^{2}\d u+\int_{\R^b}\big|\Delta^{\lceil(\beta+b)/4\rceil}\psi_{k}(u)\big|^{2}\d u\Big)\\
 &\quad\le C8^{b}(1\vee R^{\beta/2})e^{-R\delta_{n,p}^{-2}/4}\sum_{k\in\{0,j\}}\frac{C}{2^{b}\pi{}^{b/2}\Gamma(b/2)}\Big(\int_{\R^b}\big|\psi_{k}(u)\big|^{2}\d u+\int_{\R^b}\big(\Delta^{\lceil(\beta+b)/4\rceil}\psi_{k}(u)\big)^{2}\d u\Big). 
\end{align*}
We now use that $\psi_j(u)$ coincides with $\psi_0(u)$ for $|u|\ge2/\delta_{n,p}$ and thus
\begin{align*}
  &\int_{\R^b}\big(\Delta^{\lceil(\beta+b)/4\rceil}\psi_{j}(u)\big)^{2}\d u+\int_{\R^b}\big(\Delta^{\lceil(\beta+b)/4\rceil}\psi_{0}(u)\big)^{2}\d u\\
  \le& 2\int_{|u|>2/\delta_{n,p}}\big(\Delta^{\lceil(\beta+b)/4\rceil}\psi_{0}(u)\big)^{2}\d u +\int_{|u|\le2/\delta_{n,p}}\big(\big|\Delta^{\lceil(\beta+b)/4\rceil}\psi_{j}(u)\big|^{2}+\big|\Delta^{\lceil(\beta+b)/4\rceil}\psi_{0}(u)\big|^{2}\big)\d u.
\end{align*}
Since $\psi_0$ is the characteristic function of a stable distribution, the first integral can be estimated by Proposition~\ref{prop:tailbounds}. The multiplicative perturbation of $\psi_0$ in the second integral is a smooth function which can be uniformly bounded by $D^b$ for some suitable constant $D>0$. An analogous argument applies to the $L^2$-norms of $\psi_k$. Since the ball $\{u\in\R^b:|u|\le 2/\delta_{n,p}\}$ has Lebesgue measure $(2/\delta_{n,p})^b\frac{2\pi^{b/2}}{b\Gamma(b/2)}$, we get
\begin{align*}
 & \xi_{b}\int_{\R^{b}}|x|^{\beta+b}\big(f_{j}(x)-f_{0}(x)\big)^{2}\d x \\
 &\qquad\le C8^{b}(1\vee R^{\beta/2})e^{-{R}\delta_{n,p}^{-2}/4}\Big(1+\frac{c}{2^{b}\pi^{b/2}\Gamma(b/2)}\Big(\Gamma(b/2)8^b\pi^{b/2}+D^b(2/\delta_{n,p})^b\frac{2\pi^{b/2}}{b\Gamma(b/2)}\Big) \Big)\\
 &\qquad\le C8^{b}(1\vee R^{\beta/2})e^{-{R}\delta_{n,p}^{-2}/4}\Big(1+C4^b+\frac{2CD^b\delta_{n,p}^{-b}}{Tb\Gamma(b/2)^2}\Big).
\end{align*}
The last term is uniformly bounded in $b$ thus that we finally have
\[
  \xi_{b}\int_{\R^{b}}|x|^{\beta+b}\big(f_{j}(x)-f_{0}(x)\big)^{2}\d x \le C32^{b}(1\vee R^{\beta/2})e^{-R\delta_{n,p}^{-2}/5}.\tag*{\qed}
\]

\section*{Acknowledgements}
D. Belomestny acknowledges the financial support from the Russian Academic Excellence Project ``5-100'' and from  Deutsche Forschungsgemeinschaft (DFG) through the SFB 823 ``Statistical modelling of nonlinear dynamic processes''. M. Trabs gratefully acknowledges the financial support by the DFG research fellowship TR 1349/1-1. The work of A.B. Tsybakov was supported by GENES and by the French National Research Agency (ANR) under the grants IPANEMA (ANR-13-BSH1-0004-02) and Labex Ecodec (ANR-11-LABEX-0047). This work has been started while M.T. was affiliated to the Université Paris-Dauphine.

\bibliographystyle{apalike}
\bibliography{bib-3}

\end{document}